\documentclass[11 pt]{amsart}

\usepackage{amsmath}
\usepackage{amssymb}
\usepackage{amsfonts}
\usepackage{amsthm}
\usepackage{enumerate}
\usepackage[mathscr]{eucal}
\usepackage[all]{xy}
\usepackage{graphicx}
\usepackage{colortbl}
\usepackage{bbm}
\usepackage{verbatim}

\theoremstyle{plain}
\newtheorem{thm}{Theorem}[section]
\newtheorem{lemma}[thm]{Lemma}

\newtheorem{proposition}[thm]{Proposition}

\numberwithin{equation}{section}
\newtheorem*{thmunnumbered}{Theorem}

\newcommand{\urho}{{\underline{\rho}}}

\newcommand{\vsig}{\varsigma}

\newcommand{\floor}[1]{\lfloor #1 \rfloor }

\newcommand{\Be}{\begin{equation}}
\newcommand{\Ee}{\end{equation}}
\newcommand{\Bea}{\begin{align}}
\newcommand{\Eea}{\end{align}}
\newcommand{\Beas}{\begin{align*}}
\newcommand{\Eeas}{\end{endalign*}}
\newcommand{\Benu}{\begin{enumerate}}
\newcommand{\Eenu}{\end{enumerate}}
\newcommand{\Bi}{\begin{itemize}}
\newcommand{\Ei}{\end{itemize}}

\def\intslash{\rlap{\kern  .32em $\mspace {.5mu}\backslash$ }\int}
\def\qsl{{\rlap{\kern  .32em $\mspace {.5mu}\backslash$ }\int_{Q_x}}}

\def\vth{\vartheta}

\def\floor#1{{\lfloor #1 \rfloor }}
\def\floork3{{\lfloor k/3 \rfloor }}
\def\emph#1{{\it #1 }}
\def\diam{{\text{\rm diam}}}

\def\bbone{{\mathbbm 1}}

\def\Ga{\Gamma}

\def\tx{{\tilde x}}
\def\ty{{\tilde y}}

\def\cf{{\it cf}}

\def\rank{{\text{\rm rank }}}

\def\coker{{\text{\rm Coker }}}
\def\loc{{\text{\rm loc}}}

\def\comp{{\text{\rm comp}}}

\def\inn#1#2{\langle#1,#2\rangle}

\def\meas{{\text{\rm meas}}}

\def\lc{\lesssim}

\def\eps{\varepsilon}
\def\ep{\epsilon}

\def\ka{\kappa}
            
\def\la{\lambda}

              \def\Om{\Omega}
\def\ups{\upsilon}

\def\fS{{\mathfrak {S}}}

\def\ft{{\mathfrak {t}}}

\def\fw{{\mathfrak {w}}}
\def\fx{{\mathfrak {x}}}
\def\fy{{\mathfrak {y}}}
\def\fz{{\mathfrak {z}}}

\def\bbH{{\mathbb {H}}}

\def\bbR{{\mathbb {R}}}

\def\bbZ{{\mathbb {Z}}}

\def\cA{{\mathcal {A}}}

\def\cC{{\mathcal {C}}}

\def\cE{{\mathcal {E}}}

\def\cI{{\mathcal {I}}}
\def\cJ{{\mathcal {J}}}

\def\cL{{\mathcal {L}}}
\def\cM{{\mathcal {M}}}
\def\cN{{\mathcal {N}}}

\def\cR{{\mathcal {R}}}

\def\cT{{\mathcal {T}}}

\def\y{{\hbox{\roman y}}}

\def\be#1{\begin{equation}\label{#1}}
\def\endeq{\end{equation}}
\def\endal{\end{align}}
\def\bas{\begin{align*}}
\def\eas{\end{align*}}
\def\bi{\begin{itemize}}
\def\ei{\end{itemize}}

\begin{document}
\title[$L^p$-Sobolev regularity of generalized Radon transforms]
{$\mathbf L^{\mathbf p}$-Sobolev estimates for a class of integral operators with folding canonical relations}
\author{Malabika Pramanik and Andreas Seeger}

\address{Malabika Pramanik\\Department of Mathematics\\University of British Columbia\\Room 121, 1984 Mathematics Road\\Vancouver, B.C., Canada V6T 1Z2} \email{malabika@math.ubc.ca}

\address{Andreas Seeger   \\Department of Mathematics\\ University of Wisconsin-Madison\\Madison, WI 53706, USA}\email{seeger@math.wisc.edu}

\dedicatory{In memory of Eli Stein}


\begin{abstract} We prove a sharp $L^p$-Sobolev regularity result for a class of generalized Radon transforms for  families of curves in a three dimensional manifold, with folding  canonical relations.  The proof relies on  decoupling inequalities by Wolff and by Bourgain-Demeter for plate decompositions of thin neighborhoods of cones.
	\end{abstract}

\thanks{Supported in part by NSERC and NSF grants.}


\maketitle

\section{Introduction}In this paper we continue the study \cite{pr-se} of $L^p$ 
regularity properties of integral operators along families of curves in $\bbR^3$ satisfying suitable
curvature and torsion conditions.
The  previous article dealt  with the
translation invariant case, i.e. the integrals
\Be\label{translinv}
\cA f(x) =\int f(x-\gamma(s)) \chi(s) ds
\Ee
where $\gamma$ is a curve in $\bbR^3$ with nonvanishing curvature and 
torsion and $\chi$ is smooth and compactly supported. The authors showed an optimal result with a gain of $1/p$ derivatives for sufficiently large $p$, namely that for large $p$  the operator $\cA$ 
maps $L^p(\bbR^3)$ 
into the $L^p$-Sobolev space $L^p_{1/p}$.
The usual combination of damping of oscillatory integrals arguments 
and  improved $L^\infty$ bounds, as employed in \cite{SoSt},
does not apply to averaging operators 
for  curves in three or higher dimensions. Instead the authors had 
to apply a deep  result of Wolff \cite{Wolff1} on decompositions of 
cone multipliers in $\bbR^3$ which is now known as an $\ell^p$-decoupling inequality. The result 
in \cite{pr-se} can be combined with a  recent result by Bourgain and Demeter \cite{bourgain-demeter} which extends  the  decoupling result for the cone in $\bbR^3$  to the optimal $L^p$ range $p>6$; this combination immediately yields  $\cA: L^p(\bbR^3)\to L^p_{1/p}(\bbR^3)$ for 
$p>4$. A result by Oberlin and Smith \cite{Oberlin-Smith99}  shows that this range is optimal, up to possibly the endpoint $p=4$.

In the current work we shall treat extensions of these results for operators which are not of convolution type. Let $\Omega_L$, $\Omega_R$ be three-dimensional smooth manifolds and consider families of curves $\cM_x\subset \Omega_R$ parametrized by 
and smoothly depending on $x\in \Omega_L$.
 Let $d\sigma_x$ be arclength  measure on $\cM_x$ and $\chi_\circ\in C_c^\infty (\bbR^3\times \bbR^3)$.
 We define 
 the generalized Radon transform operator 
$\cR: C^\infty_c(\Omega_R)\to C^\infty(\Omega_L)$  by 
\begin{equation}
\mathcal Rf(x) = \int_{\mathcal M_x} 
f(y)  \chi_\circ(x,y) d\sigma_x(y)\,.\ \label{GRT}  \end{equation} 

In order to formulate our results we use the double fibration formalism of Gelfand and Helgason (see e.g. \cite{guillemin-sternberg}, p. 340 ff.).
 Assume 
$$
\mathcal M_x = \left\{y \in \Omega_R : (x,y) \in \mathcal M \right\}
$$
where $\mathcal M$ is  a submanifold of $\Omega_L \times \Omega_R$ of 
codimension $2$  such that the 
projections 
\begin{equation}\label{singsupp}
\xymatrix{& \mathcal M \ar[dl]\ar[dr] & \\
\Omega_L & & \Omega_R}  \end{equation}
have surjective differentials.
The surjectivity assumption on the differential of  $\cM\to \Omega_L$ implies that the  $\cM_x$ are smooth immersed curves in $\Omega_R$ (depending smoothly on $x$). Similarly 
the corresponding  assumption on the differential of  $\cM\to \Omega_R$ implies  that 
$\mathcal M^y = \left\{x \in \Omega_L : (x, y) \in \mathcal M \right\}$
 are smooth immersed curves in $\Omega_L$ (depending smoothly on $y$).

The  operator $\mathcal R$ can be realized as a Fourier integral operator of order $-1/2$ belonging to the H\"ormander class 
$I^{-\frac{1}{2}}(\Omega_L, \Omega_R; (N^{\ast}\mathcal M)' )$ 
 where 
\[(N^{\ast}\mathcal M)'  = \{(x,\xi, y,\eta): (x,y, \xi,-\eta)\in N^*\cM\}\] with $N^*\cM$ the conormal bundle given by 
\[N^*\cM  
:= \left\{(x,y, \eta, \xi)) \in T^{\ast}(\Omega_L \times \Omega_R) \setminus \{0\} : (\xi, \eta) \perp T_{(x,y)} \mathcal M \right\}.  \]
(\cf. \S\ref{FIOsection}). 
 
 The  assumptions on the projections  \eqref{singsupp}
imply that $$\mathcal C: = (N^{\ast}\mathcal M)'\subset 
 (T^{\ast}\Omega_L \setminus 0_L) \times (T^{\ast}\Omega_R \setminus 0_R) 
 $$
 with $0_L$ and $0_R$ referring to the zero sections of 
the cotangent spaces $T^{\ast}\Omega_L$ and $T^{\ast}\Omega_R$,
 respectively. $\mathcal C$ is a homogenous canonical relation, i.e. if 
 $\sigma_L$ and $\sigma_R$ are the canonical two-forms on $T^{\ast}\Omega_L$ and $T^{\ast}\Omega_R$ respectively, then $\mathcal C$ is Lagrangian with respect to $\sigma_L - \sigma_R$. As is well-known from the theory of Fourier integral operators (see \cite{hoer-fio}, \cite{phongsurvey}), the
$L^2$ Sobolev  regularity properties of $\mathcal R$ are  governed by the geometry of the projections 
\begin{equation}
\xymatrix{& \mathcal C \ar[dl]_{\pi_L} \ar[dr]^{\pi_R} & \\
T^{\ast}\Omega_L  & & T^{\ast}\Omega_R } \label{arrow-diagram}  \end{equation}

Since $\cC$ is Lagrangian the differential 
$(D\pi_L)_P$ is invertible if and only if  $(D\pi_R)_P$  is invertible (\cite{hoer-fio}). 
For the  canonical relations for 
averaging operators  over   curves in  dimensions $\ge 3$ the maps $\pi_L$ (and $\pi_R$) fail to be 
diffeomorphisms, namely  for every point $(x,y)\in \cM$ there is  $P=(x,\xi,y,\eta)\in (N^*\cM)'$ 
so that  $(D\pi_L)_P$ and $(D\pi_R)_P$ are not invertible.

\subsection*{\bf Statement of the main result}
We shall assume that the only singularities $\pi_L$ and $\pi_R$ are {\it Whitney folds} and say that $\cC$ projects with two-sided fold singularities. Recall the definition from  \cite[Appendix C4]{hoer}. Given a  $C^\infty$ map $g: X\to Y$ between $C^\infty$ manifolds and $P\in X$  the Hessian 
$g''(P)$ 
is invariantly defined as a map from $\ker(g')_P$ to $\coker(g')_{g(P)}$. Then $g$ has a Whitney fold at $P$ 
if $\mathrm{dim} (\ker(g')_P)=1$,  $\dim (\coker(g')_{g(P})=1$ and the Hessian at $P$ is not equal to $0$. 
Equivalently,  $g$ is such that for every point $P\in X$, $Dg_P$ is either invertible or $g$ has a Whitney fold then
$\cL=\{ P: \det (Dg)_P\neq 0\}$ is an immersed hypersurface of $X$ and  for any vector field $V$ with
$V_P\in \ker (Dg)_P$ for all $P\in \cL$ (a ``kernel field")  we have $V(\det Dg)\neq 0$ at $P$.

\begin{thm}\label{mainthm}
Let $\cM\subset \Om_L\times\Om_R$  be a four-dimensional manifold  such that the projections \eqref{singsupp} are submersions. Assume  that  
 the only singularities of $\pi_L:(\cN^*\cM)'\to T^*\Om_L$  and
  $\pi_R:(\cN^*\cM)'\to T^*\Om_R$ 
 are Whitney folds.
 Let $\cL\subset (N^*\cM)'$ be the conic hypersurface manifold where $D\pi_L$ and $D\pi_R$ drop rank by one, and let $\varpi$ be  
the  projection of $(\cN^*\cM)'$ to the base $\cM$. Suppose  that its restriction to $\cL$,
\begin{equation}\label{varpi} \varpi: \cL\mapsto \cM \end{equation}
is a submersion.
Then $\cR$ is extends to a continuous  operator 
$$ \cR: L^p_{\comp} (\Omega_R)\mapsto  L^p_{1/p,\loc}(\Omega_L), \quad 4<p<\infty\,.
$$
\end{thm}

The conclusion means that for  any $C^\infty$-function $\ups$ compactly supported in a coordinate chart of $\Omega_L$ and  for any compact  $K\subset \Omega_R$  we have for all $L^p$ functions $f$ supported in $K$ 
$$\|\ups\,\cR f\|_{L^p_{1/p}} \le C_p(\ups,K) \|f\|_p\,.$$
Here $L^p_{s}$ is the standard Sobolev space consisting of  tempered distributions $g$ on $\bbR^3$ with $(I-\Delta)^{s/2} g\in L^p(\bbR^3)$. It is easy to see that the regularity index $s=1/p$ cannot be improved. 
As mentioned above the result fails for $p<4$, by a result in \cite{Oberlin-Smith99}. 
Regarding the hypotheses in Theorem \ref{mainthm}, one may  conjecture that  the two-sided fold assumption can be weakened to a one-sided fold assumption, i.e. that  the assumption of $\pi_R$ being  a Whitney fold can be dropped. See \S\ref{xraysubsection} for further discussion  of relevant examples, and \S\ref{lastsection} for related results.

Using  a theorem in \cite{prs} the  regularity result  can be further improved by using Triebel-Lizorkin spaces $F^s_{p,q}$, namely we have 
\Be\label{TL}\|\ups\, \cR f\|_{F_{p,q}^{1/p}} \le C_{p,q}(\ups,K) \|f\|_{F_{p,p}^0}\,,\quad 4<p<\infty,\,\, q>0,\Ee
for $f\in F_{p,p}^0$ supported in $K$, this is  further discussed in  \S\ref{endpoint-bound}. Recall that 
$F_{p,2}^s=L^p_s$ and 
$F_{p,q}^s\subset F_{p,2}^s\subset F_{p,p}^s=B_{p,p}^0$ for $q\le 2\le p$, and any $s\in \bbR$.

\medskip

{\it Notation.} We shall use the notation $A\lc B$ for $A\le CB$ with an unspecified constant $C$.

\subsection*{\it Acknowledgement} The authors thank Geoffrey Bentsen for reading a draft of this paper and providing valuable input.

\section{Generalized Radon transforms and Fourier integral representations}\label{FIOsection}
We recall some basic facts on generalized Radon transforms and Fourier integrals. By localization we may assume that the Schwartz kernel of our operator is supported in a small neighborhood of a base point $P^\circ=(x^\circ,y^\circ) \in \mathcal M$.
On the neighborhood 
the manifold 
$\mathcal M$ is  given by a defining function $ \Phi$, i.e., $\mathcal M = \{(x, y) :  \Phi(x,y) = 0 \}$, 
where $ \Phi = (\Phi^1, \Phi^2)^{\intercal}$  is a two-dimensional vector function defined on $\Omega_L \times \Omega_R$ and such that $\Phi(P_0)=0$. 
The Schwartz kernel of our operator is given by the measure
$\chi\,\delta\circ\Phi$ 
where $\delta$ is the Dirac measure in $\bbR^2$ and $\chi$ is $C^\infty$ and compactly supported near the base point which can be chosen to be the origin in $\bbR^3\times\bbR^3$.
By the Fourier inversion formula the Schwartz kernel is  an oscillatory integral distribution, formally written as
(\cite{hoer-fio}, 
\cite{guillemin-sternberg},
\cite{SoSt})
\Be\label{FIOrep}\chi(x,y) \delta \circ \Phi (x,y) = (2\pi)^{-2} \iint e^{i(\tau_1\Phi_1(x,y) +\tau_2\Phi_2(x,y))} \chi(x,y) d\tau.\Ee

Since 
the projection $\cM\to \Om_L$  is a submersion,
 the $2 \times 3$ matrix $ \Phi_y$ has rank 2, so by a linear change of variables in $y$, near $y_0$  we can assume that $\det[\nabla_{y'}\Phi^1, \nabla_{y'}\Phi^2] \ne 0$ where $y' = (y_1, y_2)^\intercal$. Then $(x, y_3)$ can be chosen as the local coordinates on $\mathcal M$, so that the equation $ \Phi(x,y) = 0$ is equivalent to 
\Be y_i= S^i(x_1,x_2,x_3,y_3), \quad i=1,2.\Ee
Since $\Phi(x, S^1, S^2,y_3)=0$ we can write 
\begin{equation} \Phi(x,y) = \sum_{i=1}^{2} ( S^i(x, y_3)-y_i) {B}_i(x, y), \label{Phi} \end{equation}
where 
\[ {B}_i(x,y) = -\int_{0}^{1}  \Phi_{y_i}\bigl(x,  S(x, y_3) + s(y' -  S(x, y_3)), y_3 \bigr) \, ds. \]
Since $ \Phi_{y_1}$ and $ \Phi_{y_2}$ are linearly independent on $\mathcal M$, by choosing the cutoff $\chi$ to be supported sufficiently close to $\mathcal M$, we can ensure that $ B_1$ and $ B_2$ are linearly independent as well. The equation (\ref{Phi}) can therefore be re-written as \[ \begin{pmatrix} 
\Phi^1(x,y)\\ \Phi^2(x,y) \end{pmatrix} 
= B(x,y) \begin{pmatrix} S^1(x, y_3)-y_1\\
 S^2(x, y_3)-y_2\end{pmatrix} \]
where $B(x,y)$ is the $2 \times 2$ invertible matrix whose column vectors are $ B_1$ and $ B_2$. Since the projection $\cM\to \Om_R$ is a submersion  the $x$-gradients $S^1_x(x,y_3)$, $S^2_x(x,y_3)$
are linearly independent. 
Now \eqref{FIOrep} can be rewritten as
\begin{equation}
\begin{aligned}\chi(x,y) \,\delta\!\circ\!\Phi(x,y) 
  &=  \chi(x,y)
  \int_{\tau \in \mathbb R^2} e^{i \langle \tau,\Phi(x,y) \rangle} d\tau  \\ &= \frac{\chi(x,y)}{|\det B(x,y)|} \iint e^{i \langle \tau,   S(x, y_3)-y' \rangle} 
  d\tau.  
\end{aligned} \label{RS} \end{equation}

Then  in a neighborhood of the reference point $P$ the canonical relation, that is the twisted conormal bundle  $(N^*\cM)'$, 
is given by
\begin{multline*}\{(x,\xi,y,\eta): y_i=S^i(x, y_3), \,\, i=1,2, \quad \xi= \tau_1 S^1_x(x,y_3)+\tau_2 S^2_x(x,y_3),
\\
\eta= (\tau_1, \tau_2, -\tau_1S^1_{y_3}(x,y_3)-\tau_2S^2_{y_3}(x,y_3))\}.
\end{multline*}
Thus  using $(x_1, x_2, x_3, \tau_1,\tau_2, y_3)$ as coordinates on  $(N^*\cM)'$  the projection 
$\pi_L: (N^*\cM)'\to T^*\Om_L$ is identified with
\Be \label{tildepiL}\tilde\pi_L: 
 (x_1,x_2,x_3, 
\tau_1,\tau_2,y_3)\mapsto (x, \tau_1 S^1_x(x,y_3)+\tau_2 S^2_x(x,y_3)).
\Ee
Then
$$
\det  D\tilde \pi_L = \det (S^1_x, S^2_x, \tau_1 S^1_{xy_3}+\tau_2 S^2_{xy_3}) = 
\tau_1 \Delta_1+\tau_2\Delta_2
$$
with
\Be \label{Deltadef}
\Delta_i(x,y_3)\equiv \Delta_i^S(x,y_3):=  \det (S^1_x, S^2_x,  S^i_{xy_3})\big|_{(x,y_3)} , \quad i=1,2.\Ee
Hence $\cL$ is the submanifold of $(\cN^*\cM)'$ consisting of $(x,\xi,y,\eta)$ such that
\begin{align*} 
  \xi&= \tau_1 S^1_x(x,y_3)+\tau_2 S^2_x(x,y_3),
\,\, \eta= (\tau_1, \tau_2, -\tau_1S^1_{y_3}(x,y_3)-\tau_2S^2_{y_3}(x,y_3)),
\\y_i&=S^i(x, y_3), \,\, i=1,2, \,\,\,\,\,\,\,\,\,\tau_1 \Delta_1(x,y_3)+\tau_2\Delta_2(x,y_3)=0
.
\end{align*}

\section{Curvature}
We shall show that the assumptions in  Theorem \ref{mainthm} imply a curvature condition on the fibers of $\cL$, as formulated by Greenleaf and the second author in \cite{GS}.

Let $\Delta_i$ be as in \eqref{Deltadef}
and $P^\circ=(a^\circ, S^1(a^\circ,b^\circ), S^2(a^\circ,b^\circ))$ be our reference point. The following preparatory observation
is based on the assumption that $\varpi$ in \eqref{varpi} is a submersion.

\begin{lemma}\label{Deltas} We have 
$$|\Delta_1(x,y_3)|+|\Delta_2(x, y_3)|\neq 0$$
for $(x,y_3)$ near $(a,b)$.
\end{lemma}

\begin{proof} 
By continuity we have to check $|\Delta_1|+|\Delta_2|\neq 0$ at  $P^\circ$.

Let $\tau^\circ\in \bbR^2\setminus \{0\} $ and let 
$\xi^\circ=\tau_1^\circ S^1_x(a^\circ,b^\circ)+\tau_2^\circ S^2_x(a^\circ,b^\circ)$. Clearly if $(a^\circ,\xi^\circ)\notin \pi_L(\cL)$ then 
$\tau_1^\circ \Delta_1(a^\circ,b^\circ)
+\tau_2^\circ \Delta_2(a^\circ,b^\circ) \neq 0$ 
and therefore we may assume that 
$(a^\circ,\xi^\circ)\in \pi_L(\cL)$, i.e. 
\[
\tau_1^\circ \Delta_1(a^\circ,b^\circ)
+\tau_2^\circ \Delta_2(a^\circ,b^\circ) = 0.\]
Let $V_L$ be a kernel field which we may write as 
\[V_L =\sum_{i=1}^2\alpha_i(x,\tau) \frac{\partial}{\partial \tau_i} 
+
\alpha_3(x,\tau) \frac{\partial}{\partial y_3} +\sum_{i=1}^3 \beta_i(x,y_3,\tau)\frac\partial{\partial x_i}\]
where $\beta_i=0$, by \eqref{tildepiL}.  We have \begin{multline*}
 V_L (\tau_1 \Delta_1+\tau_2\Delta_2)\big|_{(a^\circ,b^\circ,\tau^\circ) }
= \\
\sum_{i=1}^2 \alpha_i(a^\circ,\tau^\circ)\Delta_i(a^\circ,b^\circ) + \alpha_3(a^\circ,\tau^\circ) \sum_{i=1}^2 \tau_i^\circ \frac{\partial \Delta_i}{\partial y_3} (a^\circ,b^\circ).
\end{multline*}

We argue by contradiction and assume that \Be \label{contra}\Delta_i(a^\circ,b^\circ)=0, \,\,i=1,2.\Ee
By assumption $V_L(\det \pi_L)\neq 0$ on $\cL$. Using \eqref{contra} we get
\Be\label{foldcond}
 \tau_1 ^\circ\frac{\partial\Delta_1 }{\partial y_3} 
 (a^\circ,b^\circ) +\tau_2^\circ\frac{\partial \Delta_2}{\partial y_3} (a^\circ,b^\circ) \neq 0.
\Ee
Hence  can, for $(\frac\tau{|\tau|},x,y_3)$ near  $(\frac{\tau^\circ}{|\tau^\circ|}, a^\circ, b^\circ)$,  solve 
$\tau_1\Delta_1  +\tau_2\Delta_2 =0$ in $y_3$
and obtain a function $\fy_3(\tau_1,\tau_2)$, homogeneous of degree $0$,  so that 
$$
\tau_1\Delta_1 (x,y_3) +\tau_2\Delta_2(x,y_3) =0 
\iff y_3=\fy_3(x,\tau).$$
Implicit differentiation gives
\Be\label{impl}  \frac{\partial \fy_3}{\partial \tau_i} = -\frac{\Delta_i(x, \tau, \fy_3)}{\tau_1 \partial_{y_3}\Delta_1
+\tau_2 \partial_{y_3}\Delta_2}, \,\,\,i=1,2.
\Ee

Now since we assume that  $\varpi : \cL\to \cM$ is a submersion
the differential of the map
$(x,\tau)\mapsto (x, S^1(x, \fy_3(x,\tau)), S^2(x, \fy_3(x,\tau)), \fy_3(\tau))$ is surjective.
This implies that
$$\rank \begin{pmatrix} \partial_{y_3}S^1(x, \fy_3) \partial_{\tau_1}\fy_3 &
\partial_{y_3}S^1(x, \fy_3) \partial_{\tau_2}\fy_3
\\
\partial_{y_3}S^2(x, \fy_3) \partial_{\tau_1}\fy_3 &
\partial_{y_3}S^2(x, \fy_3) \partial_{\tau_2}\fy_3
\\
 \partial_{\tau_1}\fy_3 &
\partial_{\tau_2}\fy_3\end{pmatrix}\,=\, 1
$$
 and thus  $|\partial_{\tau_1}\fy_3| + 
|\partial_{\tau_2}\fy_3|>0$. But by \eqref{impl} this implies that at least one of the $\Delta_i(a^\circ,b^\circ)$ is nonzero, yielding a contradiction to \eqref{contra}.
\end{proof}

It will be useful to explicitly construct  a  kernel field $V_L$ in a conic neighborhood of $\cL$. Notice that $\cL=\cL^+\cup \cL^{-}$ where $\cL^{\pm}=$
\[
\{(x, \pm \rho(-\Delta_2S^1_x + \Delta_1S^2_x), S^1(x,y_3), S^2_{x,y_3}, y_3 , \tau,\pm\rho(\Delta_2 S^1_{y_3} -\Delta_1 S^2_{y_3})): \rho>0\}.
\]
We  identify $\pi_L$ with $\tilde \pi_L$ as in \eqref{tildepiL}.

\begin{lemma} 
Define $\Gamma_i(x,y_3)$, $i=1,2$, by  \begin{subequations}\label{Gadef}
\begin{align}\label{Ga1def}
\Gamma_1&= \det\begin{pmatrix} S^1_{x}&S^2_{x,y_3} & S^1_{xy_3}\end{pmatrix}\,,
\\ \label{Ga2def}
\Gamma_2&= \det\begin{pmatrix} S^1_{xy_3} &S^2_{x } & S^2_{xy_3}\end{pmatrix}\,.
\end{align}
\end{subequations} 
Let
\Be \label{VLdef}
V_L^{\pm} =
\frac{\pm |\tau|}{\sqrt{\Delta_1^2+\Delta^2_2}} 
 \Big( 
\Gamma_2(x,y_3) \frac{\partial}{\partial \tau_1}-
\Gamma_1(x,y_3) \frac{\partial}{\partial \tau_2}\Big) + 
\frac{\partial}{\partial y_3}.
\Ee
Then 
 $V_L^+$, $V_L^-$  are  kernel fields for $\tilde \pi_L$  near $\cL^+$,
 $\cL^-$, respectively.
\end{lemma}
\begin{proof}
We take $\tau=\pm \rho(-\Delta_2, \Delta_1)$ and then the assertion reduces to showing that
\Be \label{kernelfieldprop}
\Gamma_2 S^1_x-\Gamma_1S^2_x -\Delta_2 S^1_{xy_3}+\Delta_1 S^2_{xy_3} \Big|_{(x,y_3)} =0.
\Ee

Denote the left hand side by $W$. We  first observe that
$$\det \begin{pmatrix} S^1_x&S^2_x& \Delta_1S^1_x+\Delta_2S^2_x\end{pmatrix} =\Delta_1^2+\Delta_2^2
$$
which is  nonzero, by Lemma \ref{Deltas}.
We use that  three vectors $v_1, v_2, v_3\in \bbR^3$ form  a basis of $\bbR^3$ if and only if the 
vector products $v_1\wedge v_2$, $v_1\wedge v_3$, $v_2\wedge v_3$ form  a basis, and apply this fact 
to 
$\{S^1_x,\,S^2_x,\,\Delta_1 S^1_{xy_3}+\Delta_2 S^2_{xy_3}\}$. Now  $W=0$ follows by checking 
$\inn{W}{S^i_x\wedge (\Delta_1S^1_{xy_3}+\Delta_2 S^2_{xy_3})}=0,$ for $i=1,2$, and
$\inn{W}{S^1_x\wedge S^2_x}=0$.
These are  straightforward to verify.
\end{proof}

We now consider the fibers in $T^*\Omega_L$ of $\pi_L(\cL)$, namely
\begin{subequations}
\begin{multline}\Sigma_x:= \, \{ (x,\tau_1S^1_x(x,y_3)+ \tau_2S^2_x(x,y_3)):\,\\ \tau_1\Delta_1(x,y)+\tau_2\Delta_2(x,y_3)=0, |\tau|\neq 0\}.
\end{multline}
$\Sigma_x$ is a  cone which splits as $\cup_\pm \Sigma_x^\pm$ where 
\Be \label{Xi} \Sigma_x^\pm = \{\pm \rho \Xi(x,y_3):\rho> 0\}\Ee
with
\Be \Xi(x,y_3)= -\Delta_2(x,y_3) S^1_x(x,y_3)+\Delta_1(x,y_3) S^2_x(x,y_3)\,.
\Ee
\end{subequations}
Next, for $\rho>0$
\begin{subequations}
\Be\label{foldcondka} 
V_L^\pm (\tau_1\Delta_1(x,y_3)+\tau_2\Delta_2(x,y_3))\Big|_{ \tau=\pm\rho(-\Delta_2, \Delta_1)}
=
\pm \rho\kappa(x,y_3)
\Ee
where 
\Be\label{kappadef}\kappa(x,y_3)= 
\Gamma_2\Delta_1-\Gamma_1\Delta_2
+\Delta_1\Delta_{2,y_3}-\Delta_2\Delta_{1,y_3}
 \Big|_{(x,y_3)}.
 \Ee
 \end{subequations}
This quantity is nonzero, by the assumption that $\pi_L$ projects with a fold singularity.

The following lemma will be crucial to establish the curvature properties of the cones $\Sigma_x$.
\begin{lemma}  \label{curvcalc}
Let 
$\ka$ be as in \eqref{kappadef}.
Then
$$\det \begin{pmatrix}
\Xi&\Xi_{y_3}&\Xi_{y_3y_3}\end{pmatrix}\Big|_{(x,y_3)}=-[\kappa(x,y_3)]^2.
$$
\end{lemma}
\begin{proof}We have
\begin{align*}
\Xi&= -\Delta_2 S^1_x+\Delta_1S^2_x,
\\
\Xi_{y_3}&= -\Delta_{2, y_3} S^1_x+\Delta_1 S^2_x- \Delta_2 S^1_{xy_3}+\Delta_1 S^2_{xy_3},
\end{align*} and
\begin{align*}
\Xi_{y_3y_3}=&-\Delta_{2,y_3y_3}S^1_x+\Delta_{1,y_3y_3} S^2_x \\&-2\Delta_{2, y_3}S^1_{xy_3}
+2\Delta_{1, y_3} S^2_{xy_3} -\Delta_2 S^1_{xy_3y_3}+\Delta_1 S^2_{xy_3y_3}
\end{align*}
where all expressions are evaluated at $(x,y_3)$.
Also
\begin{align*}
\Xi\wedge \Xi_{y_3}=&
(\Delta_1 \Delta_{2, y_3}-\Delta_2 \Delta_{1, y_3}) (S^1_x\wedge S^2_x)
\\&+ (\Delta_1 S^2_x- \Delta_2S^1_x)\wedge 
(\Delta_1 S^2_{xy_3}-\Delta_2 S^1_{xy_3}).
\end{align*}
Define
\[E= \Delta_1\Delta_{2,y_3}- \Delta_2 \Delta_{1,y_3}.\]
Diligent computation yields 
\begin{equation}
\inn{\Xi\wedge \Xi_{y_3}}{\Xi_{y_3y_3}}=\sum_{i=1}^5A_i
\end{equation}
where
\begin{align*} 
A_1&=-2E^2,
\\
A_2&= E\big(
\Delta_1 \det \begin{pmatrix} S^1_x& S^2_x &S^2_{xy_3y_3}\end{pmatrix}
-\Delta_2 \det \begin{pmatrix} S^1_x& S^2_x &S^1_{xy_3y_3}\end{pmatrix}\big),
\\
A_3 &=(\Delta_2\Delta_{1,y_3y_3}-\Delta_1\Delta_{2, y_3y_3}) (\Delta_1\Delta_2-\Delta_2\Delta_1)=0,
\\
A_4&=
2E(\Delta_2 \Gamma_1- \Delta_1 \Gamma_2)
\end{align*}
and 
\begin{align*}
A_5=&-\Delta_2\Delta_1 \inn{S^1_x\wedge S^2_{xy_3}}{-\Delta_2 S^1_{xy_3y_3}+\Delta_1 S^2_{xy_3y_3}}
\\
&+\Delta_2^2 \inn{ S^1_x\wedge S^1_{xy_3}}{-\Delta_2 S^1_{xy_3y_3}+\Delta_1 S^2_{xy_3y_3}}
\\
&+\Delta_1^2 \inn{ S^2_x\wedge S^2_{xy_3}}{-\Delta_2 S^1_{xy_3y_3}+\Delta_1 S^2_{xy_3y_3}}
\\
&-\Delta_1\Delta_2 \inn{ S^2_x\wedge S^1_{xy_3}}{-\Delta_2 S^1_{xy_3y_3}+\Delta_1 S^2_{xy_3y_3}}\,.
\end{align*}

We rewrite the expression $A_5=  A_{5,1}+ A_{5,2}$ where
\begin{align*}
A_{5,1}&=\Delta_2^2 \det \begin{pmatrix} 
S^1_x&\Delta_1 S^2_{xy_3}&S^1_{xy_3y_3} \end{pmatrix}
-\Delta_2^2 
\det \begin{pmatrix} 
S^1_x&\Delta_2 S^1_{xy_3}&S^1_{xy_3y_3} \end{pmatrix}
\\
&- \Delta_1\Delta_2 
\det \begin{pmatrix} 
S^2_x&\Delta_1 S^2_{xy_3}&S^1_{xy_3y_3} \end{pmatrix}
+ \Delta_1\Delta_2 
\det \begin{pmatrix} 
S^2_x&\Delta_2S^1_{xy_3}&S^1_{xy_3y_3} \end{pmatrix}
\end{align*}
and
\begin{align*}
A_{5,2}&=\Delta_1^2 \det \begin{pmatrix} 
S^2_x&\Delta_1 S^2_{xy_3}&S^2_{xy_3y_3} \end{pmatrix}
-\Delta_1^2 
\det \begin{pmatrix} 
S^2_x&\Delta_2S^1_{xy_3}&S^2_{xy_3y_3} \end{pmatrix}
\\
&- \Delta_2\Delta_1 
\det \begin{pmatrix} 
S^1_x&\Delta_1 S^2_{xy_3}&S^2_{xy_3y_3} \end{pmatrix}
+ \Delta_2\Delta_1 
\det \begin{pmatrix} 
S^2_x&\Delta_2S^2_{xy_3}&S^2_{xy_3y_3} \end{pmatrix}.
\end{align*}
Now by \eqref{kernelfieldprop} we have 
\[
\Delta_1 S^2_{xy_3}-\Delta_2 S^1_{xy_3}= -\Gamma_2 S^1_x+\Gamma_1 S^2_x.
\]
We use this to simplify  $A_{5,1}$ and $A_{5,2} $ to
\begin{align*}
A_{5,1} &= \Delta_2 (\Gamma_1\Delta_2-\Gamma_2\Delta_1) \det\begin{pmatrix} 
S^1_x&S^2_x&S^1_{xy_3y_3}\end{pmatrix}\,,
\\
A_{5,2} &= -\Delta_1 (\Gamma_1\Delta_2-\Gamma_2\Delta_1) \det\begin{pmatrix} 
S^1_x&S^2_x&S^2_{xy_3y_3}\end{pmatrix}\,.
\end{align*}
We combine these formulae with the previous ones for $A_1,\dots, A_4$ and use that $A_3=0$. We get
\begin{align*} 
&\sum_{j=1}^5 A_j\, =\,  -2E^2 +2E (\Delta_2\Gamma_1-\Delta_1\Gamma_2)\,+\,
\\&
 (E+\Delta_1\Gamma_2-\Delta_2\Gamma_1) \big( \Delta_1 \det\begin{pmatrix} S^1_x &S^2_x &S^2_{xy_3y_3}\end{pmatrix}
 -\Delta_2 \det\begin{pmatrix} S^1_x &S^2_x &S^1_{xy_3y_3}\end{pmatrix}\big)\,.
 \end{align*}
 Now using
 \[\det\begin{pmatrix} S^1_x &S^2_x &S^i_{xy_3y_3}\end{pmatrix}=\Delta_{i,y_3}- \Gamma_i, \quad i=1,2,\]
 we obtain
 \begin{align*}
 \sum_{j=1}^5 A_j\, &=\,  -2E^2 +2E (\Delta_2\Gamma_1-\Delta_1\Gamma_2)\\ &\quad\,\,+\,
  (E+\Delta_1\Gamma_2-\Delta_2\Gamma_1)
  (\Delta_1\Delta_{2,y_3}-\Delta_1\Gamma_2 -\Delta_2 \Delta_{1,y_3} +\Delta_2 \Gamma_1)
  \\&= -E^2- 2E(\Delta_1\Gamma_2-\Delta_2\Gamma_1)-(\Delta_1\Gamma_2-\Delta_2\Gamma_1)^2
    \\&= -(E+\Delta_1\Gamma_2-\Delta_2\Gamma_1)^2\,.
   \end{align*}
   which gives the assertion.
\end{proof}

We now examine the curvature properties of the cone $\Sigma_x=\{\rho \Xi(y_3)\}$.
Lemma \ref{curvcalc} 
 implies that $\Xi\wedge\Xi_{y_3}\neq 0$. 
The second fundamental form at $\rho \Xi(x,y_3)$ with respect to the unit normal
$N= \frac{\Xi\wedge\Xi_{y_3}}
{|\Xi\wedge\Xi_{y_3}|}$ is given by
\[
\begin{pmatrix}\rho \inn{\Xi_{y_3y_3}}{N} 
& \inn{\Xi_{y_3}}{N}
\\
 \inn{\Xi_{y_3}}{N} &0
\end{pmatrix}\,=\,
\begin{pmatrix} \rho \inn{\Xi_{y_3y_3}}{N} 
& 0 
\\
 0 &0
\end{pmatrix}\,.\]
Now, by Lemma \ref{curvcalc}, 
\Be\label{princ} \rho \inn{\Xi_{y_3y_3}}{N} = \frac{\rho}{|\Xi\wedge \Xi_{y_3}|} \det \begin{pmatrix}
\Xi&\Xi_{y_3}&\Xi_{y_3y_3}\end{pmatrix}= \frac{-\rho \ka(x,y_3)^2}{|\Xi\wedge \Xi_{y_3}|},\Ee
and the fold condition says that  $\ka$ does not vanish. Hence   $\Sigma_x$ is a two-dimensional cone such that everywhere there is exactly  one nonvanishing principal curvature, and it is given by \eqref{princ}.

\section{Some model operators}\label{example-sect}
 The examples motivating the present paper originate from 
 problems in harmonic  analysis and  integral
 geometry. 
We list a few of them below. The notation used in each of these examples is self-contained.

\subsection{\it Averages along curves with nonvanishing curvature and
  torsion} \label{example1}
 Let $\gamma : I \rightarrow \mathbb
R^3$ be a compact space curve with nonvanishing curvature and torsion. Then,
the integral operator $\cA$ in \eqref{translinv}
is an  example of  a 
Fourier integral operator of order $-1/2$  with two-sided fold
 singularities. Clearly the projection $\varpi$ in \eqref{varpi} is a submersion. Thus we recover the result that 
 $\cA: L^p\to L^p_{1/p}$ for $p>4$ which is known by a combination of the results in \cite{pr-se} and 
 \cite{bourgain-demeter}. The theorem in this paper shows that the $L^p_\comp\to L^p_{1/p}$ estimate holds true for small variable perturbations of the translation invariant case.

\subsection{\it Restricted $X$-ray transforms in $\mathbb R^3$}\label{xraysubsection}
A restricted $X$-ray transform in $\bbR^3$ is the restriction of the X-ray transform to a
 {\it line complex}, that is,  
a three-dimensional manifold of lines.
Under a suitable {\it well-curvedness assumption} it was shown in \cite{GU} (see also \cite{GS}) that (local versions) of this operator are Fourier integral operators of order $-1/2$  for which the projection $\pi_R$ has Whitney folds. For a class of generic line complexes we have two-sided fold singularities but this is not the case for the important class satisfying Gelfand's admissibility condition (see \cite{GU}) which is relevant for invertibility of the restricted X-ray transform.
The optimal $L^2\to L^2_{1/4}$-Sobolev regularity for the latter was
obtained in  \cite {GU-comp}, and  can also be seen as a part of a result on more general Fourier integral operators with one-sided fold singularities in \cite{GS}. 

We discuss a model case. 
Let $I$ be a compact interval and $\gamma:I \rightarrow \mathbb R^2$ a smooth regular curve 
$\gamma(t)=(t, g(t))$, $t\in I$.
We assume that $\gamma$ has 
{\it  non-vanishing curvature}, i.e. 
\Be\label{nonvancurv} g''(t)\neq 0, \quad t\in I.\Ee
 Let $\beta\in (-1,1)$ and let $e_2=(0,1)\in \bbR^2$.  For $f \in C_0^{\infty}(\mathbb R^3)$ we define
\Be\label{betarestrxray} \mathcal R_{\beta}  f(x_1,x_2, t) = \chi_1(t) 
\int f(x_1+st, x_2+ s(\beta x_2 +g(t)), s) \chi_2(s)ds \Ee
where $x' = (x_1, x_2)$ and $\chi_1, \chi_2$ are smooth real-valued
functions, with $ \chi_1$ supported in the interior of $I$ and $\chi_2$ with compact support contained in $\bbR\setminus \{-\beta^{-1}\}$.

We examine the  adjoint operator which is given by 
\[ R^*_{\beta} h(x)= \int h(S^1(x,y_3), S^2(x,y_3), y_3) \chi_2(x_2)\chi_1(y_3) dy_3\]
with 
\[ S^1(x,y_3)= x_1-x_3  y_3, \qquad
S^2(x,y_3)= \frac{x_2-x_3g(y_3)}{1+x_3\beta}.\]
Then  $\begin{pmatrix}
S^1_x&S^2_x &\tau_1S^1_{xy_3}+\tau_2S^2_{xy_3}
\end{pmatrix}$ is given by 
\[
\begin{pmatrix}
1&0&0
\\
0 &(1+x_3\beta)^{-1} &0
\\
-y_3
& \frac{-g(y_3)-\beta x_2}{(1+x_3\beta)^2} 
&
-\tau_1-\tau_2
\frac{g'(y_3)}{(1+x_3\beta)^2}
\end{pmatrix}
\]
and hence 
$$\det\pi_L=\tau_1\Delta_1+\tau_2\Delta_2
= -(1+x_3\beta)^{-1} \big(\tau_1+\tau_2
\frac{g'(y_3)}{(1+x_3\beta)^2}
\big)\,.
$$ Now $\partial /\partial y_3$ is  a kernel field and the  fold condition holds iff and only if 
\eqref{nonvancurv} holds on $I$.

The cones $\Sigma_x$ are given by
\[
 \Sigma_x = \left\{ \xi \in \mathbb R^3 : \xi = \lambda (-g'(t), 1+\beta x_3, tg'(t) - g(t)-\beta x_2 , \; \lambda \in \mathbb R,\, t\in I \right\}. \]    
To check that the  cone $\Sigma_x$ has one nonvanishing curvature everywhere one verifies that 
the plane curve
$\Gamma(t) =(-g'(t), tg'(t)-g(t))$
has nonvanishing curvature.
This holds since $\Gamma_1'\Gamma_2''-\Gamma_1''\Gamma_2'(t) = -(g''(t))^2$.

In order to apply our main result one also needs to check that the projection $\pi_R $ (for the adjoint $R_{\beta}^*$)  has only fold singularities; this turns out to be the case when $\beta\neq 0$.
$\pi_R$ is given by
$$(x_1,x_2, x_3, \tau_1,\tau_2,y_3)\mapsto
(S^1(x,y_3), S^2(x,y_3), y_3, \tau_1, \tau_2, \tau_1 S^1_{y_3}+\tau_2 S^2_{y_3})
$$
and a calculation shows that 
$V_R = y_3 \frac{\partial}{\partial x_1}
+\frac{\partial}{\partial x_3}
$ is a kernel field.
Then $$V_R (\tau_1\Delta_1+\tau_2\Delta_2)= -\tau_2 g'(y_3)\frac{2\beta}{(1+\beta x_3)^3}\,.
$$
Thus, if $\beta\neq 0$, $\pi_R$ has only fold singularities. 
Now  Theorem \ref{mainthm} implies that for $\beta\neq 0$  the operator 
$R_{\beta}^* $ maps $L^p$ to $L^p_{1/p}$ for $p>4$, and more generally
$L^p_\alpha$ to $L^p_{\alpha+1/p}$.  Hence 
\Be R_{\beta}: L^p\to L^p_{1/p'} , \, 1<p<4/3, \, 
\label{xrayest}\Ee
when $\beta\neq 0$.
Our theorem does not apply to the case $\beta=0$, when $\pi_R$ has maximal degeneracy (a blowdown singularity). However by a rather straightforward argument it was shown  in \cite{Pramanik-Seeger}
that 
\eqref{xrayest} remains valid if $\beta=0$ (provided one uses the result by Bourgain and Demeter in 
 conjunction with \cite{Pramanik-Seeger}). This leads one  to   conjecture that the assumption on $D\pi_R$ in Theorem \ref{mainthm} can be dropped.
  
\subsection
{\it Averages along  curves in $\mathbb H^1$} 
Convolution operators on noncommutative groups can often be analyzed as generalized Radon transforms. Let us consider the Heisenberg group 
$\mathbb H^1$ which is $\mathbb R^3$ with the group multiplication defined by
\[x \cdot y = \bigl(x_1 + y_1, x_2 + y_2, x_3 + y_3 + \tfrac{1}{2}(x_1 y_2 - x_2 y_1)\bigr). \]

\subsubsection{Measures on curves in the plane.} 
Let $I$ be a bounded open interval and $g\in C^\infty$. We consider convolution on $\bbH^1$ with  a measure on $\bbR^2\times \{0\}$ supported on $\{(t, g(t), 0): t\in I\}$ where $g''(t)\neq 0$ for $t\in I$. For $\chi\in C^\infty_0(I)$ 
define $\mu$ by 
\[ \langle \mu, f \rangle := \int_{\gamma} f(t, g(t),0) \chi(t) dt\]
and the convolution
\[\mathcal \cA_1 f(x) := f \ast \mu(x) = \int f \bigl(x' - y', x_3 - \tfrac{1}{2}(x_1 y_2 - x_2 y_1)\bigr) d \mu_0(y'). \]
Then $\cA_1 f(x) $ can be written as 
$\int  f(y_1, S^2(x,y_1), S^3(x,y_1)) \chi(x_1-y_1) dy_1$
with 
\begin{align*}
S^2(x,y_1)&= x_2-g(x_1-y_1), \\ S^3(x,y_1)&= x_3- \frac{x_1}{2}g(x_1-y_1)+\frac{x_2}{2}(x_1-y_1).
\end{align*}
As  observed in \cite{MuSe},  $\mathcal A_1$ is a Fourier integral
operator 
with  folding canonical relation (i.e. $\pi_L$ and $\pi_R$) project with folds. Moreover 
$$\pi_L (N^*\cM)=\{(x, \tau_2 S_x^2(x,y_1)+\tau_3S_x^3(x,y_1)\}$$ and 
$\det \pi_L = g''(x_1-y_1) (\tau_2+\tau_3x_1/2)$ and thus  $\Sigma_x$ is given by the parametrization
$$\Xi(\tau_3,t) =
 \frac{\tau_3}{2} \big( x_2-g(t), -x_1+t,1\big).
 $$
 
 This example, and higher dimensional versions 
 were 
 considered in \cite{MuSe} for the $L^2$-Sobolev category, together with some refinements,  that yield sharp maximal function estimates on $\bbH^n$, $n\ge 2$. 
 The measure in the horizontal plane can also be replaced by other measures in other planes transversal to the center, in which case the estimates in \cite{MuSe} yield less satisfactory results for maximal function bounds. However in this case
  sharp $L^p$-Sobolev estimates and maximal function bounds for $n\ge 2$ have been recently established in \cite{acps}, using methods which are closely related to the current paper. For a more recent result on circular maximal functions on the Heisenberg group see \cite{bghs} where the case of Heisenberg radial functions is considered.

\subsubsection{Averages along space curves in $\mathbb H^1$} \label{momentcurveHsect}
A closely related example  was considered by
Phong and Stein \cite{PhSt}  and Secco \cite{secco}.
Let $\gamma_{\alpha} : I \rightarrow \mathbb H^1$ be the curve given by $\gamma_{\alpha}(s) = (s, s^2, \alpha s^3)$, where $\alpha$ is a real-valued parameter, and $I$ a bounded interval. Given a cutoff function $\chi \in C_0^{\infty}(I)$, let us consider the singular measure $\mu_{\alpha}$ on $\mathbb H^1$ supported on $\gamma_{\alpha}$ given by 
\[ \langle \mu_{\alpha}, f \rangle = \int_I f(\gamma_{\alpha}(s)) \chi(s) \, ds, \]
and the right convolution operator by $\mu_{\alpha}$:
\Be\label{convalphaH}
 \mathcal A_{2,\alpha} f(x) = f \ast \mu_{\alpha}(x) := \int f(x \cdot \gamma_{\alpha}(s)^{-1})\, ds, \qquad x \in \mathbb H^1. \Ee
As shown in \cite{PhSt} 
 $\mathcal A_{2,\alpha}$ is a Fourier integral operator, with
two-sided 
folds for $\alpha \ne \pm \frac{1}{6}$ and with one-sided folds for
$\alpha = \pm \frac{1}{6}$. A special role of the parameters $\pm \frac 16$  has also been observed by Secco
\cite{secco} in the context of $L^p\to L^q$ estimates. 
It is straightforward 
to verify that the 
projection $\pi_L$ in this problem is a fold if and only if 
$\alpha = \frac{1}{6}$, and the cone 
$\Sigma_x \subseteq \mathbb R^3_{\xi}$ is generated by the parabola 
\[ \xi_1 = \frac{x_2}{2} + 2(6 \alpha - 1)t^2, \quad \xi_2 = -\frac{x_1}{2} - (6\alpha - 1) t, \quad \xi_3 = 1. \]     
Our result  yields the  sharp $L^p$ regularity properties for all $\alpha\in \bbR\setminus\{\pm 1/6\}$ but it does not cover the cases $\alpha=\pm 1/6$. Bentsen \cite{bentsen}  
obtained a sharp $L^p$ regularity results 
for  a  class  of averaging operators over   curves  in the Heisenberg group for which one of $\pi_L$, $\pi_R$  is a fold and the other is a blowdown. 
It turns out that  in the case $\alpha=-1/6$ of \eqref{convalphaH} the  local regularity results follow by changes of variables directly from the regularity results  for the  restricted X-ray transform in \eqref{betarestrxray}
when $\beta=0$ (i.e. the case considered in \cite{Pramanik-Seeger}).

\section{Basic decompositions}\label{basicdecsect}
We decompose dyadically in $\tau$ (for large $\tau$).   Then for 
$|\tau|\approx 2^k$ we decompose further according to the size of $2^{-k}\det\pi_L$ which is approximately 
the size of $2^{-k}(\tau_1\Delta_1+\tau_2\Delta_2)$, 
which is also approximately the distance to the fold surface.
This decomposition is standard and goes back to
\cite{PhSt}
(with earlier precursors). 

Let $\eta_0 \in C_c^{\infty}(\mathbb R)$ be an even function so that $\eta_0(s) = 1$ for $|s| \leq \frac{1}{2}$ and supp$(\eta_0) \subset (-1, 1)$, and set $\eta_1(s) = \eta_0(\frac{s}{2}) - \eta_0(s)$. Then $\eta_0(s) + \sum_{k \geq 1}\eta_1(2^{1-k}s) \equiv 1$ for $s \geq 0$. Define
\begin{subequations} 
\begin{align}\label{defofchik}
\chi_k(x,y,\tau) &:= \chi(x,y) \eta_1(2^{1-k}|\tau|) \quad \text{ for } k \geq 1, \\ \chi_{0}(x, y, \tau) &:= \chi(x, y) \eta_0(|\tau|),
\end{align}  and, after changing variables in $\tau$ 
\Be\label{Rkdefinition}
\mathcal R_kf(x) := 2^{2k}\iint e^{i 2^{k}\langle \tau,  S(x, y_3)-y' \rangle} \chi_k(x,y,2^k\tau)d\tau\, f(y) dy ,
\Ee   
 \end{subequations} 
with  $\langle \tau,  S(x, y_3)-y' \rangle=\sum_{i=1}^2\tau_i(S^i(x,y_3)-y_i)$ and now $|\tau|\approx 1$. We then have  $$\mathcal Rf = \sum_{k \geq 0} \mathcal R_k f$$
for all Schwartz functions $f$. For $0 \leq \ell \le \lfloor k/3\rfloor
$, let \[
\chi_{k,\ell}(x, y, \tau) := \begin{cases}\chi_k(x,y,2^k\tau) \eta_1 \bigl(2^{\ell}(\tau_1 \Delta_1 + \tau_2 \Delta_2) \bigr), \text{ if } \ell < \lfloor k/3\rfloor, 
\\
\chi_k(x, y, 2^k\tau) \eta_0 \bigl(2^{\lfloor\frac{k}{3}\rfloor } (\tau_1 \Delta_1 + \tau_2 \Delta_2) \bigr), \text{ if } \ell = \lfloor k/3\rfloor, \end{cases} \] and
\begin{subequations}\label{Rkldef}
\begin{align}
R_{k, \ell}(x, y) &:=2^{2k}  \int e^{i 2^k\langle \tau, y' - S(x, y_3) \rangle} \chi_{k, \ell}(x, y, 2^k\tau) \, d\tau,  \\
\mathcal R_{k, \ell}f(x) &:= \int  R_{k, \ell}(x, y) f(y) \, dy.
\end{align}
\end{subequations}
so that $\mathcal R_{k} = \sum_{\ell \leq \frac{k}{3}} \mathcal R_{k, \ell}$. 
For $k>0$ the $\tau$-integration is extended over a subset of the annulus $\{1/2<|\tau|<2\}$ (indeed the intersection of this annulus with a $C2^{-\ell}$-neighborhood of a line $l(x,y_3)$).
 The quantity  $\tau_1\Delta_1+\tau_2\Delta_2$, when $|\tau|\approx 1$  is comparable to  the  distance to the fold surface $\cL$.

The by now standard $L^2$ estimate for the 
operators $\cR_{k,\ell}$ is
\Be\label{phstest}
\|\cR_{k,\ell}\|_{L^2\to L^2}  \lc 2^{\frac{\ell-k}2}
\Ee
for $\ell=0,1,\dots \lfloor k/3\rfloor$, see  \cite{cuccagna}.
The following estimates 
will be  the main ingredient for the proof of Theorem \ref{mainthm}.

\begin{thm}  \label{Rklthm} Let $0<\epsilon<1/6$. 
For $\ell \leq \lfloor k/3\rfloor$ 
we have
\begin{equation}\label{Rlkthmest} \|\mathcal R_{k, \ell}\|_{L^p\to L^p}  \leq C_{\epsilon,p} \cdot \begin{cases}
2^{\ell(\epsilon +\frac 2p-\frac 12)} 2^{- \frac{k}{p}},  \quad &4<p\le 6,
\\
2^{-\ell(\frac {1-6\epsilon}p)} 2^{- \frac{k}{p}}, \quad &6\le p\le\infty.
\end{cases}\end{equation}
\end{thm}

The endpoint Sobolev bound will follow from this theorem  with some 
additional arguments, see  \S\ref{endpoint-bound}.

The  main important tool in the proof is the  following decoupling inequality.

\begin{thm} \label{decouplingthm}
Let $\ell\le k/3$ and let $\epsilon>0$.
Let, for $\nu\in \bbZ$ 
\Be\label{fnu}  f_\nu(y)= f(y) \bbone_{[2^{-\ell}\nu, 2^{-\ell}(\nu+1)]}(y_3).\Ee
Then for $2\le p\le 6$,
\Be\label{decouplingest}
\Big\| \sum_\nu \cR_{k,\ell} f_\nu
\Big\|_p 
\le C_\epsilon 2^{\ell (\epsilon+\frac 12-\frac 1p)}
\Big(\sum_\nu \big\|\cR_{k,\ell} f_\nu \big\|_p^p\Big)^{1/p} + C_{\epsilon} 2^{-k} \|f\|_p.
\Ee
\end{thm}

Theorem \ref{decouplingthm} will be proved by induction, see \S\ref{decouplingstepsect}. In each induction step we will combine a standard application of the Wolff-Bourgain-Demeter decoupling theorem  in combinations with suitable changes of variables.
\
\begin{proof}[Proof that Theorem \ref{decouplingthm} implies Theorem \ref{Rklthm}]

We first note that  for $g_\nu\in L^\infty$ and with
$\bbone_{\nu,\ell}(y_3) :=\bbone_{[2^{-\ell}\nu, 2^{-\ell}(\nu+1)]}(y_3)$
\Be\label{infty} \sup_\nu \|\cR_{k,\ell} [\bbone_{\nu,\ell} g_\nu] \big\|_\infty \lc 2^{-\ell} \sup_\nu  \|g_\nu\|_\infty.
\Ee
To see this one derives an estimate for the Schwartz kernel $R_{k,\ell}(x,y)$ by integrating by parts, distinguishing the directions $(\Delta_1,\Delta_2)$ and $(-\Delta_2,\Delta_1)$.
This shows that $|R_{k,\ell}(x,y)| \le C_N \prod_{i=1}^2
U_{k,\ell,i}(x,y)$ where
\begin{align*}U_{k,\ell,1}(x,y)&=
 \frac{2^{k-\ell} }
{(1+2^{k-\ell} |\Delta_1(y_1-S^1)+\Delta_2(y_2-S^2)|)^N}
\\
U_{k,\ell,2}(x,y)&=
\frac{2^k}
{(1+2^{k} |-\Delta_2(y_1-S^1)+\Delta_1(y_2-S^2)|)^N}
\end{align*}
where $S^1, S^2,\Delta_1, \Delta_2$ are evaluated at $(x,y_3)$. We integrate in $(y_1,y_2)$ first and then use that the $y_3$ integration is extended over an interval of length $2^{-\ell}$. This yields \eqref{infty}.

From \eqref{phstest} and averaging with Rademacher functions we also get 
\Be\label{two} \Big(\sum_\nu \big\|\cR_{k,\ell} [\bbone_{\nu,\ell} g_\nu] \big\|_2^2\Big)^{1/2} \lc 2^{\frac{\ell-k}2} \Big(\sum_\nu  \|g_\nu\|_2^2\Big)^{1/2}, \notag
\Ee
and by interpolation, 
\Be\Big(\sum_\nu \big\|\cR_{k,\ell} [\bbone_{\ell,\nu}g_\nu] \big\|^p\Big)^{\frac 1p} 
\lc 2^{\ell (\frac 3p-1)} 2^{-\frac kp}\Big(\sum_\nu  \|g_\nu\|_p^p\Big)^{\frac 1p}, \quad 2\le p\le \infty. \label{ellpLpbds}
\Ee 
Combining this with \eqref{decouplingest}
we obtain 
\Be
\big\|  \cR_{k,\ell} f
\big\|_p 
\lc C_\eps 2^{\ell (\epsilon+ \frac 2p-\frac 12)} 2^{-\frac kp}
\|f\|_p, \quad\text{  $2\le p\le 6$.}
\notag
\Ee
Finally from \eqref{infty} we also have the  bound $\|\cR_{k,\ell}\|_{L^\infty\to L^\infty} =O(1)$ and a further interpolation gives the inequality asserted in \eqref{Rlkthmest} for $6\le p\le \infty$.
\end{proof} 
 
\subsection*{\it An estimate in Besov-spaces}
Theorem \ref{Rklthm} implies an estimate in Besov spaces. To see that
we let $L_k$ be the operator defined by $\widehat {L_k f}=\beta(2^{-k} \xi)\widehat f$ 
where $\beta\in C^\infty_c (\bbR^3\setminus \{0\})$. Integration by parts arguments show that there exists a constant $C$ such that 
\Be\label{Littlewood-Paley-calc}
\|L_{k'} \cR_{k,\ell} L_{k''}\|_{L^p\to L^p} \le C_N 2^{-N \max\{k,k',k"\}} \text{ if } 
\max\{|k-k'|, |k-k''|\}>C
\Ee
whenever $\min\{k,k', k''\}\ge 3\ell$.
This, together with the main estimate
\Be \|\cR_k\|_{L^p\to L^p} \le C(p)  2^{-k/p}\text{ for } p>4.\Ee
implies the boundedness result 
\[ \cR:\, (B^s_{p,q})_{\text{comp}}  \to (B^{s+1/p}_{p,q})_{\text{loc}}, \text{ for } p>4.\]
For the more sophisticated Sobolev bounds, and improvements, see \S\ref{endpoint-bound}.

\section{The decoupling step in a model case}\label{decouplingstepmodelsect}

In this section we consider a model version of the operator $\cR_{k,\ell}$ defined in \eqref{Rkldef},
where the functions $S^i$ are replaced by $\fS^i$ 
satisfying additional assumptions at the origin, see 
\eqref{Sinormal} below. These normalizing assumptions will enable us to carry out a decoupling step as suggested by the Bourgain-Demeter decoupling theorem which we  review in \S\ref{BDsect}.
The reduction of the general case to the model case will be carried out later in \S \ref{decouplingstepsect}, using suitable changes of variables discussed in \S\ref{changesofvarsect}.

\subsection{\it The Bourgain-Demeter decoupling theorem}\label{BDsect}
Let $\kappa_0\neq 0$ be a constant.  We use the decoupling result in \cite{bourgain-demeter}, for the part of the cone  
\[ \Sigma=\{\xi: \ka_0  \xi_2\xi_3+\tfrac 12\xi_1^2 =0\}\] where $|\xi_2| \approx 1$, $|\xi_1|\ll 1$. A parametrization  is given by
$$ \xi(b,\la)= \lambda (-\ka_0 b e_1+e_2-\tfrac 12 \ka_0 b^2e_3)
$$  
where $|\la|\approx 1$, $|b|\ll |b|M_0\ll 1$.
Let \Be
T_1(b)= \frac \partial{\partial \la} \xi(b,\la)=
-\ka_0 be_1+e_2-\tfrac 12\ka_0 b^2 e_3, \label{T1def}\Ee
 be the tangent vector pointing towards the origin and let 
\Be \widetilde T_2(b)= -\ka_0^{-1} \la^{-1} \frac \partial{\partial b} \xi(b,\la)
=
e_1+be_3.\label{T2tildedef}
\Ee
Then $T_1(b)$ and $\widetilde T_2(b)$ form a basis of the tangent space of $\Sigma$  at $\la\xi(b)$.
A normal vector is given by 
\Be
N(b)= T_1(b)\wedge \widetilde T_2(b)=
be_1+\tfrac 12\ka_0  b^2e_2- e_3.\label{Ndef}
\Ee
For the definition of our plate we need to replace $\widetilde T_2(b)$ by a vector in the span of $T_1(b)$ and $\widetilde T_2(b)$ that is perpendicular to $T_1(b)$.
Such a vector is given by 
\begin{align}
T_2(b)&= (1-\frac 14\ka_0^2 b^4) e_1+ \ka_0 b(1+\tfrac 12 b^2) e_2+ (b+\tfrac 12 \ka_0^2 b^3)e_3\label{T2def}
\\
&=e_1 +\ka_0 be_2+ be_3 +  O(b^3) \notag
\end{align}

Let $A>1$. For $\delta\ll 1$  
let  
\begin{multline}\label{platedef}
\varPi_{A,b}(\delta)\,=\,\\
\big\{\xi\in \bbR^3:\,\,
A^{-1}\le |\inn {\tfrac{T_1(b)}{|T_1(b)|}}{\xi} |\le A,\quad 
|\inn {\tfrac{T_2(b)}{|T_2(b)|}}{\xi} |\le A\delta,\quad 
|\inn {\tfrac{N(b)}{|N(b)|}}{\xi} |\le A\delta^2\big\}.
\end{multline}

One refers to the sets $\varPi_{A,b} (\delta)$ as plates; they are  unions of $A(1,\delta,\delta^2)$-boxes 
with the long, middle, short side parallel to $T_1(b)$, $T_2(b)$, $N(b)$, respectively.

\begin{thmunnumbered}[\cite{bourgain-demeter}] Let $\epsilon>0$, $A>1$. There exists a constant $C(\epsilon,A)$ such that the following holds for  $0<\delta_1<\delta_0<1$.

Let $B=\{b_\nu\}_{\nu=1}^M$ be a set of points  in $[-1,1]$ such that $|b_\nu-b_{\nu'}|\ge \delta_1$ for $b_\nu, b_{\nu'}\in B$, $\nu\neq \nu'$, and $B$ is contained in an interval of length $\delta_0$.
Let $2\le p\le 6$. Let $f_\nu\in L^p(\bbR^3)$ such that the Fourier transform of $f_\nu$ is supported in $\varPi_{A,b_\nu}(\delta_1)$. Then
\Be\label{BD} 
\Big\|\sum_\nu f_\nu \Big\|_p \le C(\epsilon, A)(\delta_0/\delta_1)^{\epsilon} \Big(\sum_\nu\|f_\nu\|_p^2\Big)^{1/2}.
\Ee
\end{thmunnumbered}
One also has
\Be\label{Wolffineq} 
\Big\|\sum_\nu f_\nu \Big\|_p \le C(\epsilon, A) 
\begin{cases} (\delta_0/\delta)^{\epsilon+1/2-1/p}
\Big(\sum_\nu\|f_\nu\|_p^p\Big)^{1/p}, &p\le 6,
\\
(\delta_0/\delta)^{\epsilon+1-4/p}
\Big(\sum_\nu\|f_\nu\|_p^p\Big)^{1/p}, & 6\le p \le\infty.
\end{cases}
\Ee
This  is the $\ell^p$-decoupling result that was first proved for large $p$ by Wolff \cite{Wolff1}. 
\eqref{Wolffineq} follows from \eqref{BD}  by H\"older's inequality and interpolation arguments.
Our  proof of Theorem \ref{mainthm} will be  based on \eqref{BD} but could also be based on the case $p>6$ of 
\eqref{Wolffineq}, as it was in (\cite{pr-se}), in the case of convolution operators.   A variant of this argument was also given  in the  manuscript \cite{pr-se-unpublished} on the variable case, an unpublished  precursor to  the current paper, with only a preliminary result.

\subsection{\it The model case} \label{modelcasesect} 
For $i=1,2$ consider $C^\infty$ functions $(w,z_3)\mapsto \fS^i(w,z_3)$ defined on a neighborhood $U$ of $[-r,r]^4$, for some $r\in (0,1)$. Assume that $M_0$  satisfies
\Be \label{M0eps0}M_0\ge 2+\|\fS^1\|_{C^5([-r,r]^4)}
+\|\fS^2\|_{C^5([-r,r]^4)}, 
\Ee where the $C^5$ norm is the maximum of the supremum of all derivatives of order $0,\dots,5$.
We assume that for $w\in [-r,r]^3$
\begin{subequations}\label{Sinormal}
\Be\label{Sinormalabc}
   (\fS^1, \fS^2, \fS^1_{z_3})\big|_{(w,0)}= (w_1,w_2, w_3);
   \Ee
moreover
\Be
\label{Sinormal-d}
    \fS^2_{wz_3}(0,0)= 0, 
\Ee
and
\Be\label{Sinormale}
 \fS^2_{w_3z_3z_3}(0,0)= \kappa_0.
   \Ee   
\end{subequations}  

Let in \eqref{defofchik} the function $\chi_0$ be  supported in a neighborhood $V$ of $(0,0)\in \bbR^3\times\bbR^3$ which is of diameter $\le 10^{-10} M_0\ll r$ and 
let $(w,z)\mapsto \alpha(w,z)$ be a $C^\infty$ function satisfying 
\Be\label{alphaequivalence}M_0^{-1} \le |\alpha (w,z)|\le M_0\Ee 
and with the higher derivatives of $\alpha$ depending on $M_0$ and the order of differentiation.

Let $(w,z_3,\mu)\mapsto \zeta(w,z_3,\mu)$ belong to a bounded family of $C^\infty$ functions supported where $-r\le w_i, z_3\le r$ and $1/4\le|\mu|\le 4$.
Let $\eta$ be $C^\infty$ and supported in $(-2,2)$ and let $\cT_{k,\ell}$ be the operator with Schwartz kernel
\begin{multline}\label{tildeRkl} 
\cT_{k, \ell}(w,z) :=2^{2k} 
 \int_{\bbR^2} e^{i 2^k\langle \mu,  \fS(w, z_3) -z'\rangle} 
\times \\\eta \bigl(2^{\ell}\alpha(w,z)(\mu_1 \Delta_1^\fS(w,z_3) + \mu_2 \Delta_2^{\fS}(w,z_3)) \bigr)
\zeta(w, z,\mu) \, d\mu\,.  \end{multline}
Here   
 $\Delta_i^\fS(w,z_3)=\det(\fS^1_w, \fS^2_w, \fS^i_{wz_3})$.
We shall omit the superscript 
and   assume throughout this
subsection \S\ref{modelcasesect} 
that $\Delta_i\equiv\Delta_i^\fS$.
The operator $\cT^{k,\ell}$ is  a version of $\cR^{k,\ell}$ defined before 
under the additional assumptions in \eqref{Sinormal}. 
We need to include the function $\alpha$ in the localization to provide added flexibility in the later stages of the proof of Theorem \ref{decouplingthm} when we apply repeated changes of variables
(\cf. formula \eqref{Deltarelation} below).

The basic decoupling step is summarized in
 \begin{proposition}	
  \label{decouplingstepmodelprop}
  Let $0<\epsilon<1/2$. There is a constant $C_\epsilon$ so that the following holds.
  
  Let  $\ell\le \floor{k/3}$ and let 
 \begin{subequations} \label{delta0delta1assu}\Be\label{delta0,1}\delta_0, \delta_1 \in (M_0^2 2^{20-\ell (1-\epsilon^2)}, 2^{-\ell\epsilon^2-20} M_0^{-2})\Ee  such that \Be \label{deccond}
 2^{100} M_0\max \{ (2^{-\ell} \delta_0)^{1/2}, \delta_0^{3/2} \}< \delta_1<\delta_0.\Ee
 \end{subequations}
  Let 
 $\cI_J$ be a collection of intervals of length $\delta_1$ which have disjoint interior and which are contained in $[0,\delta_0]$.
 Let $a\in \bbR^3$, $\vsig\in C^\infty_c$ supported in $(-1,1)^3$ and $\vsig_{\ell,0}(w)=\vsig (2^{\ell}w).$
 Then
for $2\le p\le 6$, for $g\in L^p(\bbR^3)$ and 
  $g_I(y):=g(y)\bbone_I(z_3)$ we have
\Be\label{decouplingestmodel}
\Big\| \vsig_{\ell,0} \sum_{I\in \cI_J} \cT_{k,\ell} g_I\Big\|_p 
\le C_\epsilon (\delta_0/\delta_1)^{\epsilon}
\Big(\sum_{I\in \cI_J} \big\|\vsig_{\ell,0}\cT_{k,\ell} g_I \big\|_p^2\Big)^{1/2} + C_{\epsilon} 2^{-10 k} \|g\|_p.
\Ee
\end{proposition}

In order to apply \eqref{BD} in this situation we need to consider the Fourier transforms of $\vsig_{\ell,0} \sum_{I\in \cI_J} \cT_{k,\ell} g_I$ and show that they are concentrated on the plates $2^k\varPi_{A,b_I} (\delta_1)$ for $b_I\in I$ and suitable $A>1$. We establish this plate localization in \S\ref{platelocalization} and conclude the proof of Proposition \ref{decouplingstepmodelprop} in \S\ref{proofofprop}.

\subsection{\it Derivatives of $\fS$ and $\Delta$}\label{fSderivatives}
We use this section to record some facts needed later  in \S\ref{platelocalization}, about various derivatives of $\fS^i(w,z_3)$ and $\Delta_i(w,z_3)$, under the assumption that
\Be\label{wz_3assu}|w|_\infty\le 2^{-\ell}\le\delta_0, \quad |z_3|\le \delta_0,\Ee
under  the specifications in \eqref{delta0delta1assu}.

\subsubsection{Taylor expansion of $\fS^1_w$ and $\fS^2_w$}
\begin{lemma}
Let $w$, $z_3$, $2^{-\ell}$, $\delta_0$ be as in \eqref{wz_3assu}. Then
\Be\label{fSwexpansions}
\begin{aligned}
\fS^1_w(w,z_3)&= e_1+z_3 e_3 + E^1(w,z_3)
\\
\fS^2_w(w,z_3)&= e_2+ \tfrac 12 \ka_0 z_3^2 e_3+ E^2(w,z_3)
\end{aligned}
\Ee
where
\Be\label{E1bound}
|\inn{e_i}{E^1(w,z_3)}|\le 8M_0\delta_0^2, \quad i=1,2,3,
\Ee
and\begin{subequations}
\begin{align}\label{E2bound12}
|\inn{e_i}{E^2(w,z_3)}|&\le 8M_0\delta_0^2, \quad i=1,2,
\\
\label{E2bound3}
|\inn{e_3}{E^2(w,z_3)}|&\le M_0(8 \delta_02^{-\ell}+2\delta_0^3),
\end{align}
\end{subequations}
\end{lemma}
\begin{proof}
We expand using conditions \eqref{Sinormalabc} 
and obtain 
\begin{align*} 
\fS^1_w(w,z_3)&= e_1+ z_3 e_3 + \widetilde E^1(w,z_3)
\\
\fS^2_w(w,z_3)&= e_2+  \widetilde E^2(w,z_3)
\end{align*}
where for $\nu=1,2$ we have  
$\inn{e_i}{\widetilde E^\nu}= I_{i,\nu}+II_{i,\nu}+III_{i,\nu}$ with
\begin{align*}
I_{i,\nu}(w,z_3)&= \int_0^1 (1-s) \sum_{j=1}^3\sum_{k=1}^3 \fS^\nu_{w_iw_jw_k}(sw,sz) w_jw_k \, ds\\
II_{i,\nu}(w,z_3)&=
\int_0^1 (1-s)\, 2\sum_{j=1}^3 
\fS^\nu_{w_iw_jz_3} (sw,sz) w_jz_3\, ds\\
III_{i,\nu}(w,z_3)&= \int_0^1 (1-s) \fS^\nu_{w_iz_3z_3} (sw, sz)z_3^2\,ds
\end{align*}
and obtain the bounds
\begin{align*}
|I_{i,\nu}(w,z_3)|&\le \tfrac 92 M_0 |w|_\infty^2\le \tfrac 92 M_0 2^{-2\ell}
\\ 
| II_{i,\nu}(w,z_3)| &\le \tfrac 62 M_0|w|_\infty |z_3| \le 3M_0 2^{-\ell}\delta_0
\\
| III_{i,\nu}(w,z_3)| &\le \tfrac 12 M_0|z_3|^2 \le \tfrac 12 M_0 \delta_0^2
\end{align*}
Recall $\kappa_0=\fS^2_{w_3z_3z_3} (0,0)$. 
For $i=3$, $\nu=2$ we expand further 
\[III_{3,2}(w,z_3)= \tfrac 12 \ka_0 z_3^2 +E_{3,2}(w,z_3)\]
where
\[|E_{3,2}(w,z_3)|\le M_0\big( \tfrac 32 |w|_\infty|z_3|^2+ \tfrac 12 |z_3|^3\big) \le 2M_0\delta_0^3
\]
(where we used $2^{-\ell}\le \delta_0$). Combining terms  we obtain the stated error estimates.
\end{proof}

\subsubsection{Computations involving $\Delta_1$ and $\Delta_2$}
By the assumption $\delta_0\le 2^{-10}M_0^{-1}$ we have from  above
$|\fS^1_{w_i}(w,z_3)|\le 2$, $|\fS^2_{w_i}(w,z_3)|\le 2$ and
$|\fS^1_{w_iz_3}(w,z_3)|\le 2$. 
Moreover $|\fS^2_{wz_3}(w,z_3)|\le 4M_0\delta_0 \ll 1$.
Using upper bounds for $\fS^i_w$, $\fS^i_{wz_3}$ and higher derivatives,
 the  permutation formula for determinants, trilinearity of the  determinants and differentiation of products we see that
any first order partial derivative  of $\pm \Delta_i$ is 
a sum of $3\cdot 6$ terms, each bounded by $4M_0$.
Hence  
any first order partial derivative  of $\pm \Delta_i^\fS$ is 
bounded by $72M_0$, and similarly, by the structure of the $\Delta_i$, 
any second  order partial derivative  of $\pm \Delta_i^\fS$ is 
bounded by $216M_0$.
Moreover,
any third   order partial derivative  of $\pm \Delta_i^\fS$ is 
bounded by $54\cdot 2M_0^2$. These observations also yield
$$|\Delta_1(w,z_3)-1|, |\Delta_2(w,z_3)| \le 72M_0\delta_0 \le 2^{-10}.$$

In \S\ref{platelocalization} we  shall use a Taylor expansion and rely on 
the conditions 
 \eqref{Sinormal}. This yields
\[\Delta_1(w,0)=1,\quad \Delta_2(w,0)=0,\]
and straightforward computations 
give
\begin{align*} \Delta_{1,z_3} (w,0)&= \fS^1_{w_3z_3z_3}(w,0)+\fS^2_{w_2z_3} (w,0)
\\
\Delta_{2,z_3}(w,0)&= \fS^2_{w_3z_3z_3}(w,0)+ \fS^2_{w_1z_3} (w,0)
\end{align*} and thus
\[\Delta_{1,z_3} (0,0)= \fS^1_{w_3z_3z_3}(0,0),\quad
\Delta_{2,z_3}(0,0)= \fS^2_{w_3z_3z_3}(0,0)=\kappa_0.
\]
Further  calculations  give
\begin{align*}
\Delta_{1, z_3z_3}(0,0)&= 3 \fS^1_{w_1z_3z_3}(0,0)+ \fS^2_{w_2z_3z_3} (0,0)+ \fS^1_{w_3z_3z_3z_3}(0,0)
\\
\Delta_{2, z_3z_3} (0,0)&=2 \fS^2_{w_1z_3z_3}(0,0)+ \fS^2_{w_3z_3z_3z_3}(0,0),
\end{align*}
\begin{align*}
\Delta_{1, w_jz_3} (0,0)&= \fS^2_{w_2w_jz_3}(0,0)+ \fS^1_{w_3w_jz_3z_3}(0,0)
\\
\Delta_{2, w_jz_3} (0,0)&= \fS^2_{w_3w_jz_3z_3}(0,0),\end{align*}
and
\begin{align*}
\Delta_{1, w_jw_k} (0,0)&=  0
\\
\Delta_{2, w_jw_k} (0,0)&= \fS^2_{w_3z_3w_jw_k}(0,0).\end{align*}

\subsection{\it Plate localization in the model case} 
\label{platelocalization} The following lemma contains the  information that will allow us to apply the decoupling inequality \eqref{BD}.
\begin{lemma} \label{platelocalizationlemma}
Let $\delta_0$, $\delta_1$ be as in \eqref{delta0delta1assu}.
Let $2^{-\ell}\ll r$, $M_02^{-\ell} \le 2^{-10}$, $w\in [-2^{-\ell}, 2^{-\ell}]$, $|z_3|\le \delta_0$. 
Suppose $1/4<|\mu|\le 4$ and 
\Be\label{detcond}\big|\mu_1\Delta_1(w,z_3)+\mu_2\Delta_2(w,z_3)\big|\le M_02^{-\ell}.
\Ee
Then \Be\mu_1 \fS^1_w(w,z_3) + \mu_2 \fS^2_w(w,z_3) 
\in \varPi_{A,b}(\delta_1), \quad A=2(1+|\ka_0|).
\Ee
\end{lemma} 
\begin{proof}

We examine the quantity
$\mu_1 \fS^1_w+\mu_2\fS^2_w$,  for $1/4<|\mu|\le 4$ and under the condition \eqref{detcond},
and rewrite it as
\Be \label{twoterms}\frac{1}{ \Delta_1} \Big( (\mu_1\Delta_1+\mu_2\Delta_2)
\fS^1_w
+   \mu_2(\Delta_1 \fS^2_w- \Delta_2 \fS^1_w)\Big).
\Ee
The assumption \eqref{detcond} and $|\mu|\in (1/4,4)$ implies that $|\mu_1|\le 2^{-8}$ and hence $|\mu_2|\in (1/5,4)$.

The second expression  in \eqref{twoterms} is the main term for our analysis. 
We  use a  Taylor expansion:
\begin{align}\label{Taylorexpansion}
\Delta_1 \fS^2_w-\Delta_2 \fS^1_w&\Big|_{(w,z_3)}
=  e_2 
+ v_0 z_3 +\sum_{j=1}^3 v_j x_j
\\&+ \frac{1}{2} \Big(v_{0,0} z_3^2
+ 2\sum_{j=1}^3 v_{0,j} z_3 w_j
+ \sum_{j=1}^3\sum_{k=1}^3  v_{j,k}  w_j w_k\Big)
+ \cE(w,z_3)\notag
\end{align}
where $\cE(w,z_3)$ is the Taylor reminder which vanishes of third order. 
Since $\Delta_1(w,0)=1$, $\Delta_2(w,0)=0$  the leading term is $e_2$.
For the $\bbR^3$-valued coefficients of the linear term we get 
(with all terms on the right hand side evaluated at $0$, and using  input from \S\ref{fSderivatives})
\begin{align*}
v_{0}&=\Delta_{1, z_3}\fS^2_w+\Delta_1 \fS^2_{wz_3} 
-\Delta_{2,z_3} \fS^1_w- \Delta_2 \fS^1_{wz_3}\Big|_{0,0}
\\&= \fS^1_{w_3z_3z_3}(0,0)e_2-\kappa_0 e_1
\end{align*}
and, for $j=1,2,3,$
\begin{align*}
v_j&=\Delta_{1,w_j} \fS^2_w+\Delta_1 \fS^2_{ww_j}
-\Delta_{2,w_j} \fS^1_w-\Delta_2 \fS^1_{ww_j}\Big|_{(0,0)}\,=\,0.
\end{align*}
For the coefficients of the quadratic terms we have
\begin{align*}
v_{0,0}& =\Delta_{1,z_3z_3} \fS^2_w+2 \Delta_{1,z_3} \fS^2_{wz_3} +\Delta_1 \fS^2_{wz_3z_3}
\\&\quad-\Delta_{2,z_3z_3} \fS^1_w-2 \Delta_{2,z_3} \fS^1_{wz_3} -\Delta_2 \fS^1_{wz_3z_3}
\Big|_{(0,0)}
\\
=& (\fS^2_{w_1z_3z_3}-\Delta_{2,z_3z_3})e_1
+(\fS^2_{w_2z_3z_3}+\Delta_{1,z_3z_3})e_2 
+ (\fS^2_{w_3z_3z_3}-2\Delta_{2,z_3})e_3\Big|_{(0,0)};
\end{align*}
in particular
\Be\label{quadraticmainterm} \inn{v_{0,0}}{e_3} =-\ka_0.\Ee
Moreover, for $j=1,2,3$, 
\begin{align*}
v_{0,j}=& \Delta_1\fS^2_{ww_jz_3} +\Delta_{1,w_j} \fS^2_{wz_3}+\Delta_{1,z_3}\fS^{2}_{ww_j}+\Delta_{1, w_jz_3} \fS^2_w
\\
&-
\Delta_2\fS^1_{ww_jz_3} -\Delta_{2,w_j} \fS^1_{wz_3}-\Delta_{2,z_3}\fS^{1}_{ww_j}-\Delta_{2, w_jz_3} \fS^1_w\Big|_{(0,0)}
\\
=& (\fS^2_{w_1w_jz_3}-\Delta_{2,w_jz_3})e_1+(\fS^2_{w_2w_jz_3}+\Delta_{1,w_jz_3})e_2+
\fS^2_{w_3w_jz_3} e_3\Big|_{(0,0)},
\end{align*}
and, for $j,k=1,2,3$,
\begin{align*}
v_{j,k}=& \Delta_1\fS^2_{ww_jw_k} +\Delta_{1,w_j} \fS^2_{ww_k}+\Delta_{1,w_k}\fS^{2}_{ww_j}
+\Delta_{1, w_jw_k} \fS^2_w
\\
&-
\Delta_2\fS^1_{ww_jw_k} -\Delta_{2,w_j} \fS^1_{ww_k}-\Delta_{2,w_k}\fS^{1}_{ww_j}-\Delta_{2, w_jw_k} \fS^1_w\Big|_{(0,0)}
\\
=& \Delta_{1,w_jw_k}(0,0)e_2-
\Delta_{2,w_jw_k}(0,0)e_1.
\end{align*}

Gathering terms in the above Taylor expansion 
leads to 
\begin{align}\label{Taylorexpansionsimplified}
&\Delta_1 \fS^2_w-\Delta_2 \fS^1_w\Big|_{(w,z_3)}
=e_2   -\ka_0 z_3 e_1 -
\tfrac 12 \ka_0  z_3^2 e_3
\\
&   \,\,\, +\fS^1_{w_3z_3z_3}(0,0)z_3 e_2+\sum_{j=1}^3 \fS^2_{w_3w_jz_3}(0,0) w_jz_3 e_3+\sum_{i=1}^2 r_i(w,z_3) e_i +\cE_3(w,z_3)\notag
\end{align}
where we get 
\begin{subequations}\label{errorsest} 
\Be \label{errorone} 
|\fS^1_{w_3z_3z_3}(0,0)z_3|\le M_0\delta_0, 
\Ee and by assumption \eqref{deccond},
\Be\label{errortwo}
\sum_{j=1}^3| \fS^2_{w_3w_jz_3}(0,0) w_jz_3 |
 \le 3M_0 2^{-\ell} \delta_0 \ll \delta_1^2.
\Ee
For the quadratic error terms in the first two coordinates we have
\Be\label{riest}
|r_i(w,z_3) |\le  8M_0\delta_0^2\   ,\quad i=1,2,
\Ee
and finally for the cubic error terms we have the straightforward  estimate 
\Be\label{cubicerror} |\cE_3(w,z_3)|\le 2^{20} M_0^2\delta_0^3.\Ee
\end{subequations} 

 Now consider the situation where $|z_3-b|\le \delta_1$. 
 Let $T_2(b)$ be as in \eqref{T2def}, that is, 
 $T_2(b)=e_1 +\ka_0 be_2+ be_3 +  O(b^3)$ with $1/2\le |T_2(b)|\le 2$. We  compute
\begin{equation} \label{widthintangdir}
\inn{\tfrac{T_2(b)}{|T_2(b)|}}{\Delta_1 \fS^2_w-\Delta_2 \fS^1_w}\Big|_{(w,z_3)}
= \tfrac{1}{|T_2(b)|} \kappa_0 (b-z_3)
 +\cE_{T_2}(w,z_3)
 \end{equation}
where (\cf. \eqref{deccond})
\[|\cE_{T_2}(w,z_3)| \le 2^{13} M_0\delta_0^2\le 
2^{13} M_0 (2^{-100}M_0^{-1} 
\delta_1)^{4/3}\ll\delta_1.\]

The computation for the normal component is more subtle. With $N(b)=
be_1+\tfrac 12\ka_0  b^2e_2- e_3$, we consider the contributions of the terms in the above Taylor expansion to 
$\inn{N(b)}{\Delta_1 \fS^2_w-\Delta_2 \fS^1_w}$.
We get
\begin{align*}
\inn{N(b)}{\Delta_1 \fS^2_w-\Delta_2 \fS^1_w}&=\tfrac 12\ka_0  b^2 -\ka_0b z_3
+\tfrac 12 \ka_0 z_3^2 
\\
& +
\tfrac 12\ka_0 b^2 z_3 \fS^1_{w_3z_3z_3}(0,0)-
\sum_{j=1}^3 \fS^2_{w_3w_jz_3}(0,0) w_jz_3 
\\&
+br_1(w,z_3)+ \tfrac 12\ka_0 b^2 r_2(w,z_3) + \inn{N(b)}{\cE_3(w,z_3)}.
\end{align*}
By \eqref{errorsest},
\begin{align}\label{widthinnormaldir}
&\inn{\tfrac{N(b)}{|N(b)|}}
{\Delta_1 \fS^2_w-\Delta_2 \fS^1_w}=
\tfrac{1}{|N(b)|}\big(
\tfrac 12\kappa_0 (z_3-b)^2 + \cE_N(w,z_3)\big)
\\
&\text{ with }  |\cE_N(w,z_3) |\le 2^{21} M_0^2\delta_0^3 +2^{2-\ell}M_0\delta_0\ll \delta_1^2
\notag\end{align}
where for the error estimate we have used \eqref{deccond}.
Clearly the main term on the right hand side is $\le |\ka_0|\delta_1^2/2$.

This finishes the analysis of the second term in \eqref{twoterms}.
Finally consider the first term in  \eqref{twoterms}, again under the assumption \eqref{detcond}. We get the estimates
\[
 | (\mu_1\Delta_1+\mu_2\Delta_2) \inn{\fS^1_w(w,z_3)} {T_i(b)}|\le 10M_0 2^{-\ell} \ll \delta_1, \quad i=1,2\]
and
\begin{align*}
 &| (\mu_1\Delta_1+\mu_2\Delta_2) \inn{\fS^1_w(w,z_3)} {N(b)}|\\&\le  2^{-\ell}|\inn{e_1+z_3 e_3} {be_1-e_3}| +10^2 M_0^2 2^{-\ell}\delta_0^2 \le 2^{2-\ell} \delta_1 \ll \delta_1^2.
\end{align*}
The proof is completed by combining terms. \end{proof}

\subsection{\it Proof of the decoupling step in the model case}\label{proofofprop}

Fix $b$ and let $m_{k,\delta_1, b}$ be a multiplier that is equal to $1$ on $\varPi_{2A,b}(\delta_1)$ and equal to $0$
on $\bbR^3\setminus \varPi_{3A,b}(\delta_1)$, and satisfies the natural differentiability properties
\begin{multline*}
\big|\inn{T_1(b)}{\nabla}^{\alpha_1}
\inn{T_2(b)}{\nabla}^{\alpha_2}
\inn{N(b)}{\nabla}^{\alpha_3} m_{k,\delta_1, b}(\xi)\big|
\\
\lc_{\alpha} 2^{-k\alpha_3} (2^{k}\delta_1)^{-\alpha_2} 
(2^{k}\delta_1^2)^{-\alpha_3} 
.\end{multline*} Let $P_{k,\delta_1, b} $ be defined by 
$\widehat{P_{k,\delta_1, b} f}= m_{k,\delta_1, b} \widehat f$. Let $I$ be an interval of length $\delta_1$ and let $f_I(y)=f(y)\bbone_{I}(y_3)$. The Schwartz kernel of  \[f\mapsto (I-P_{k,\delta,b})[\vsig_{\ell,0}\cT_{k,\ell} f]\] is given as a sum of  oscillatory integrals $\sum_{n=0}^\infty K_{n,k,\ell} $ where for $n>0$
\begin{multline*}K_{n,k,\ell}(w,z)\,=\,
2^{2k} 
\iiint 
  e^{i (\inn{w-v}{\xi} +
  2^k\langle \tau,  \fS(v, z_3) -z'\rangle}\vsig_{\ell,0}(v)
  \times \\
(1-m_{k,\delta_1,b}(\xi)) 
\eta_1(|\xi|2^{-n})
 \chi_{k, \ell}(v, z, 2^k\tau)  dv \, d\xi d\tau\,\,\bbone_I(z_3),
\end{multline*}
with
\[\chi_{k,\ell}(v,z,2^k\tau):=
\eta \bigl(2^{\ell}\alpha(w,z)(\tau_1 \Delta_1^\fS(w,z_3) + \tau_2 \Delta_2^{\fS}(w,z_3)) \bigr)
\zeta(w, z,\tau,k) \, d\tau\,\]
and the family $\zeta(\cdot,\cdot,\cdot,k)$ is bounded  uniformly in $C_c^\infty$.
If $|n-k|> 10$,  then  repeated integration by parts 
in the $v$-variables (followed by subsequent integration by parts in the $\xi$-variables) 
shows that
\[|K_{n,k,\ell}(w,z) |\lc \min \{2^{-10 n}, 2^{-10 k}(1+|w-z|)^{-N}\},\quad |k-n|\ge C.
\]
For $|k-n|\le C$ a similar argument applies to  the 
assumption that on the support of $(1-m_{k,\delta_1,b})$ we have $2^{-k}\xi \notin \varPi_{3A,b}(\delta_1)$.
That means 
\[\big|\,\nabla_v[-\inn{v}{\xi} +\inn{2^k\tau}{\fS(v,z_3)}]\,\big|\ge c 2^k\delta_1^2.\]
Differentiating the amplitude gives a factor of $2^\ell$ with each differentiation.
Thus for $|k-n|\le C$ an $N$-fold integration by parts
in the $v$ variables followed by integration by parts in  the $\xi$-variables shows that
\[|K_{n,k,\ell}(w,z) |\lc_N (2^k\delta_1^2 2^{-\ell})^{-N}(1+|w-z|)^{-N_1},\quad |k-n|\le C.
\]
Notice that by $\ell\le k/3$ and $\delta_1\ge 2^{-(1-\eps^2)\ell} $ we have \[2^k\delta_1^2 2^{-\ell}
\ge 2^{k\eps^2 /3}.
\]
Thus a $\floor{40/\eps^2}$-fold integration by parts in $v$ (again followed by multiple integration by parts in $\xi$) yields 
\[|K_{n,k,\ell}(w,z) |\lc  2^{-11 k} (1+|w-z|)^{-N_1}.\]
Let $b_I$ be the left endpoint of the interval $I$.
We decompose the left hand side of \eqref{decouplingestmodel} as 
\Be
\label{maindecompdecjunk}
\Big\| 
\sum_{I\in \cI_J} P_{k,\delta_1,b_I} [\vsig_{\ell,0} \cT_{k,\ell} g_I]\Big\|_p 
+\Big\| \sum_{I\in \cI_J} (I-P_{k,\delta_1,b_I} )[\vsig_{\ell,0} \cT_{k,\ell} g_I]\Big\|_p 
\Ee
By Lemma \ref{platelocalizationlemma} we can apply the decoupling inequality \eqref{BD} (with $\epsilon$ replaced by $\epsilon^2$)  to bound  the first term in \eqref{maindecompdecjunk} 
by
\begin{align*} &C(\eps^2, A) \delta^{-\eps^2} \Big(\sum_{I\in \cI_J}
\big \| P_{k,\delta_1,b_I} [\vsig_{\ell,0} \cT_{k,\ell} g_I]\big\|_p^p\Big)^{1/p}
\\&\lc
C(\eps^2, A) \delta^{-\eps^2} \Big(\sum_{I\in \cI_J}
\big \| \vsig_{\ell,0} \cT_{k,\ell} g_I\big\|_p^p\Big)^{1/p}
\end{align*}
For the second term in \eqref{maindecompdecjunk} we use the above error estimates,  apply Minkowski's inequality and get the bound
\eqref{maindecompdecjunk}  by
\begin{align*}
&2^{-11k}\sum_{I\in \cI_J} \Big(\int\Big|\int (1+|w-z|)^{-N} |g(z)\bbone_{I}(z_3)| dz\Big|^p dw\Big)^{1/p} \lc 2^{-10 k} \|g\|_p.
\end{align*}
This finishes the proof of Proposition \ref{decouplingstepmodelprop}.
\qed

\section{\it Families of changes of variables}
\label{changesofvarsect}
Let $P^\circ=(a^\circ, y^\circ)\in \cM$,
with $y^\circ=S^1(a^\circ, b^\circ),S^2(a^\circ, b^\circ),b^\circ)$.
For $r>0$ let 
\[Q(r):=\{(x,y_3): |x-a^\circ|_\infty\le r\} \text{ and }
I(r):=\{y_3: |y_3-b^\circ|\le r\}.  \]

Let $S^i$ be smooth functions in a neighborhood of $Q(2r_0)\times I(2r_0)$, for some $r_0>0$.
After possibly permuting the variables $y_1$, $y_2$ we may assume, by Lemma \ref{Deltas} that
$\Delta_1(x,y_3)=\det (S^1_x, S^2_x, S^1_{xy_3})\neq 0$ on  $Q(2r_0)\times I(2r_0))$.
Choose $M$ so that  \[M> 2+\|S\|_{C^5(Q(2r_0)\times I(2r_0))}+ \max_{(x,y_3)\in Q(2r_0)} |\Delta_1(x,y_3)|^{-1}.\]

We now consider $(a,b)$ close to $(a^\circ, b^\circ)$
and construct changes of variables so that  in the new coordinates  theconstant coefficient decoupling theorem in Proposition \ref{decouplingstepmodelprop} can be applied at suitable scales. The idea of applying a constant coefficient decoupling  theorem in a variable coefficient situation also appears in \cite{bhs}. 

For $a\in Q(2r_0)$, $b\in I(2r_0)$ let  $\Gamma_1$, $\Gamma_2$ be as in \eqref{Gadef}, and let   $\rho\equiv \rho(a,b)\in \bbR^3$ be defined by
\Be\label{defofrho}
(\rho_1,\rho_2,\rho_3)= \frac{1}{\Delta_1(a,b)}\big( -\Gamma_2(a,b), 
\Gamma_1(a,b), 
\Delta_2(a,b)\big) 
\Ee

For $(x,y_3),(a,y_3) \in Q(r_0)$ and $(a,y_3)
\in I(2r_0)$
consider the function 
\begin{align*}
&(x,a,y_3)\mapsto \fw(x,a,y_3)
\\
&Q(r_0)\times Q(r_0)\times I(2r_0)\to \bbR^3
\end{align*} defined by

\Be\label{fwdef}
\begin{pmatrix}
\fw_1\\ \fw_2\\ \fw_3\end{pmatrix}\,=\,
\begin{pmatrix}  S^1(x,b)-S^1(a,b),
\\
S^2(x,b)-\rho_3(a,b)S^1(x,b) -
S^2(a,b)+\rho_3(a,b)S^1(a,b) 
\\
S^1_{y_3}(x,b)-S^1_{y_3}(a,b)
\end{pmatrix}.
\Ee

We  have 
\Be\label{Detwx}\det (D\fw(x,a,b)/Dx)
= \det (S^1_x, S^2_x-\rho_3 S^1_x, S^1_{x,y_3})|_{(x,b)}=\Delta_1(x,b).\Ee
By the implicit function theorem there exists $r_1>0$ with $r_1<r_0$ such that for $|w|_\infty<2r_1$, $|a-a^\circ| <2r_1$, $b-b_\circ|<2r_1$ the equation
$\fw(x,a,b)=w$ is solved by a unique  $C^\infty$ function \Be
\label{xdefin}x=\fx(w,a,b).\Ee

Note the estimate \Be\label {rhoiupperbd}
|\rho_i(a,b)|\le 6M^4, \text{ for $a\in Q(2r_0)$, $b\in I(2r_0)$.}\Ee By the definition of $\fw$ and the mean value theorem for the coordinate functions this implies $|\fw(x,a,b)|_\infty\le 3M(1+6M^4)|x-a|_\infty$
for $x,a\in Q(r_0)$, $b\in I(2r_0)$.
Hence if $r_2<r_1$ and if $|x-a^\circ|_\infty<r_2$ and
$|a-a^\circ|_\infty<r_2$ then $|\fw(x,a,b)|_\infty \le 42 M^5 r_2$ 
and if we define \Be\label{r2def}r_2= (50 M^5)^{-1}r_1\Ee we get 
$|\fw(x,a,b)|_\infty <r_1$ for $x,a\in Q(r_2)$, $b\in I(2r_1)$. By the uniqueness of the function $\fx$ we thus see that $\fx(\fw(x,a,b))=x$ for $x,a \in Q(r_2)$
for $x,a\in Q(r_2)$, $b\in I(2r_1)$.

We will also need to change variables in the $y$-variables, in a more explicit form.
Define 
 \Be\label{fzdef}\fz=(\fz_1,\fz_2,\fz_3): \bbR^2\times Q(2r_0)\times I(2r_0) \to \bbR^3\Ee by 
\[\fz_1(y,a,b)= y_1-S^1(a,y_3), \qquad \fz_3(y,a,b)=y_3-b,\] and
\begin{align*}
&\fz_2(y,a,b)=\\ &y_2-\rho_3(a,b) y_1 -S^2(a,y_3)+\rho_3 (a,b)S^1(a,y_3) - (y_3-b) \sum_{i=1}^2 \rho_i(y_i-S^i(a,y_3)).
\end{align*}
We  have 
\Be\label{Detzy}\det (D\fz/Dy)= (1-\rho_2 (y_3-b)).\Ee 
By \eqref{rhoiupperbd} this quantity lies in $(1/2,3/2)$ provided that $y_3,b\in I(2r_3)$ with 
\Be\label{r3def} r_3<\min\{r_1, (24M^4)^{-1}\}.\Ee

The inverse $z\mapsto \fy(z,a,b)$, defined for $|z_3|\le r_3$, $|b-b^\circ|\le r_3$, $|a-a^\circ|\le 2r_0$,  is given by
\Be\label{fydef}\begin{aligned}
\fy_1(z,a,b)&= z_1+ S^1(a,b+z_3),
\\
\fy_2(z,a,b)&= \frac{z_2+z_1(\rho_3(a,b)+\rho_1(a,b)z_3)+ (1-z_3) S^2(a, b+z_3)}{1-\rho_2(a,b)z_3},
\\
\fy_3(z,a,b)&=b+z_3.
\end{aligned}
\Ee

\begin{lemma} The functions $\fx$, $\fy$ defined above have the following properties.

(i) $\fx(0,a,b)=a$, $\fy (0,a,b)= (S^1(a,b), S^2(a,b),b)$, $\fy_3(z,a,b)=b+z_3$.

(ii) $\det \big(\frac{D\fx(w,a,b)}{Dw} \big) = \frac{1}{\Delta_1(\fx(w,a,b),b)}.$

(iii) Let $\rho\equiv \rho(a,b)$ be as in \eqref{defofrho}
and let
\Be \label{defofB} B(z_3,a,b)= \begin{pmatrix}
1&0\\
-\rho_3-\rho_1 z_3& 1-\rho_2z_3
\end{pmatrix}.\Ee 
 Then for $|z_3|\le r_3$, $|a-a^\circ|_\infty\le r_2$, $|w|\le r_2$
  \Be\label{SversusfrakS}
B(z_3,a, b)
\begin{pmatrix} S^1(\fx(w,a,b),b+z_3)-\fy_1(z,a,b)
\\ S^2(\fx(w,a,b),b+z_3)-\fy_2(z,a,b)\end{pmatrix}=
\begin{pmatrix}
\fS^1(w,z_3, a,b)-z_1
\\ \fS^2(w,z_3,a,b)-z_2\end{pmatrix} 
\Ee
where $\fS^i$ are $C^\infty$ with
\Be \label{fSpureder} \fS^1(w,0)=w_1, \quad \fS^2(w,0)=w_2, \quad \fS^1_{z_3}(w,0)=w_3;
   \Ee
   moreover 
   \Be \label{fScondsecondder}\fS^2_{wz_3}(0,0,a,b)=0.
   \Ee
  
(iv) Let \begin{align*}
\Delta_i^S(x,y_3)&= \det (S^1_x, S^2_x, S^1_{xy_3})|_{(x,y_3)}, \\ \Delta_i^\fS(w,z_3)&= \det (\fS^1_w, \fS^2_w, \fS^1_{wz_3})|_{(w,z_3)}.  \end{align*}
Then, for 
$(\tau_1,\tau_2)= (\mu_1,\mu_2) B(z_3,a,b),$
\Be \label{Deltarelation}\sum_{i=1}^2 \tau_i\Delta_i^S(\fx(w,a,b), b+z_3)= \frac{\Delta_1^S (\fx(w,a,b),b)}
{1-\rho_2(a,b) z_3}\,
\sum_{i=1}^2\mu_i\Delta_i^\fS(w,z_3).
\Ee 

(v) Let  $\ka$  be  as in \eqref{kappadef}. Then
\Be\label{curvature}\fS^2_{w_3z_3z_3} (0,0, a,b)= \frac{\kappa(a,b)}{\Delta_1(a,b)^2}.
\Ee
\end{lemma}

\begin{proof}
We write for $i=1,2$
\begin{align*} 
S^i(x,y_3)-y_i&=  S^i(a,b) + S^i(x,b)-S^i(a,b)
\\&+ S^i(x,y_3)-S^i(a,y_3)-S^i(x,b)+S^i(a,b)
\\&+ S^i(a,y_3)-S^i(a,b)-y_i
\end{align*}
and set 
\begin{align*}
\tilde x_i&= S^i(x,b)-S^i(a,b), \quad i=1,2, 
\\
\tilde x_3&= S^1_{y_3}(x,b)-S^1_{y_3}(a,b)
\end{align*}
so that
\[\det \big(\frac{D\tilde x}{Dx} \big) = \Delta_1^S(x,b).
\]
Also let 
\begin{align*}
\tilde y_i&= y_i- S^i(a,y_3), \quad  i=1,2.
\\
\tilde y_3&=y_3-b
\end{align*}
 Note that
 $\tilde x^\intercal = (x-a)^\intercal A^\intercal +O(|x-a|^2)$ where 
 $A^\intercal$ is the matrix with column vectors $(S^1_x, S^2_x, S^1_{xy_3})|_{(a,b)} $.
We then expand 
\begin{align*}
S^1(x,y_3)-y_1&= \tilde x_1-\tilde y_1+ \tilde x_3\tilde y_3 + R_{1,1} (\tilde x,\tilde y_3,a,b) 
\\
S^2(x,y_3) -y_2&= \tilde x_2-\tilde y_2 + \tilde y_3\sum_{i=1}^3 \underline{\rho}_i\tilde x_i + R_{2,1}(\tilde x, \tilde y_3,a,b)
\end{align*}
where
\Be\label{gammai}
\urho_i= \inn {A^{-1}e_i}{ S^2_{x,y_3}(a,b)},
\Ee
and
where $R_{1,1}$, $R_{2,1}$ vanish to third order with no pure $\tx$ or pure $\ty_3$ terms, moreover  $\partial_{\ty_3} R_{1,1}$ has no pure $\tx$ terms. We label $R_{1,1}$ an  error term of {\it type $I$} and 
$R_{2,1}$ an error term of {\it type $II$}. Precisely, an error term of type $I$ is of the form
\begin{subequations}\Be
 \label{cubicI}\tilde y_3^2 \sum_{i=1}^3\tilde x_i\tilde \beta_i(\tilde x,\tilde y_3, a,b) 
\Ee
with $\tilde \beta_i$ smooth,
and a term of type $II$ is of the form
\Be \label{cubic2}
y_3 \sum_{j=1}^3\sum_{k=1}^3 \tilde x_j\tilde x_k \tilde \beta_{jk}(\tilde x,\tilde y_3,a,b) \,+\, \text{term of type $I$},
\Ee
with  $\tilde \beta_{jk} $ smooth. \end{subequations}
 Note that  $(\urho_1,\urho_2,\urho_3)$ 
satisfies 
\[
\urho_1 S^1_x(a,b)+\urho_2 S^2_x(a,b)+ \urho_3 S^1_{xy_3}(a,b) = S^2_{xy_3}(a,b) 
\]
and hence, by Cramer's rule, we see that 
$\Delta_1(\urho_1,\urho_2,\urho_3)= (-\Ga_2,\Ga_1, \Delta_2)$, i.e. 
\[\urho_i=\rho_i\] where $\rho_i$ is as in \eqref{defofrho}.

Given  $c_1, c_2,c_3\in \bbR$ we compute
\begin{align*}
&(c_3 +c_1\ty_3) (S^1(x,y_3)-y_1) + (1+c_2 \ty_3)(S^2(x,y_3)-y_2)
\\&= (\tx_2+c_3\tx_1)-(\ty_2+c_3\ty_1)
\\ &- \ty_3(\sum_{i=1}^3 \tx_i(\rho_i+c_i))
-c_1\ty_1\ty_3 -c_2\ty_2\ty_3 + R_{2,2}(\tx, \ty_3,a,b)
\end{align*}
where $R_{2,2}$ is an error term of type $II$.
We choose $c_i=-\rho_i(a,b)$ 
so that the mixed quadratic terms drop out.

We now change variable in $\tilde x$ and in $\tilde y$ separately,  setting
\[
z_1=\ty_1,\quad z_2= \ty_2-\rho_3\ty_1-\rho_1\ty_1\ty_3-\rho_2\ty_2\ty_3, \quad z_3=\ty_3
\]
and 
\[ 
w_1=\tx_1, \quad w_2= \tx_2-\rho_3 \tx_1,\quad w_3=\tx_3.\]
Define
\[\fS^i(w,z_3,a,b)=S^i(\fx(w,a,b), b+z_3), \quad i=1,2.\]
Setting
\[ B(z_3,b)= \begin{pmatrix}
1&0\\
-\rho_3-\rho_1 z_3& 1-\rho_2z_3
\end{pmatrix}\] 
 we obtain   
 \Be\label{SversusfrakSdef}
B(y_3-b,b)
\begin{pmatrix} S^1(x,y_3)-y_1
\\ S^2(x,y_3)-y_2\end{pmatrix}=
\begin{pmatrix}
\fS^1(w,z_3, a,b)-z_1
\\ \fS^2(w,z_3,a,b)-z_2\end{pmatrix} 
\Ee
if $w=\fw(x,a,b)$ and $y=\fy(z,a,b)$ and $\fw$ and $\fy$ are as in 
\eqref{fwdef} and  \eqref{fydef}. Now
\eqref{SversusfrakSdef} implies  \eqref{SversusfrakS}.

The functions $\fS^1$, $\fS^2$ satisfy
 \begin{align*}\fS^1(w,z_3, a,b)&= w_1+ w_3z_3 + R_{1,3} (w,z_3,a,b)
 \\
 \fS^2(w,z_3, a,b)&= w_2+  R_{2,3} (w,z_3,a,b)
 \end{align*} where $R_{1,3}$ is an error term of type $I$ (with $(\tilde x,\tilde y_3)$ replaced by
  $(w,z_3)$, cf. \eqref{cubicI})  and $R_{2,3} $ is a term of type $II$ (again in the $(w,z_3)$-variables, \cf. \eqref{cubic2}).
 We see that  \eqref{SversusfrakS} and  \eqref{fSpureder},  
 \eqref{fScondsecondder}
  hold.

In order to obtain \eqref{Deltarelation} we calculate
\begin{multline*}\sum_{i=1}^2 (B^\intercal \mu)_i \Delta_i^S(\fx(w,a,b), b+z_3)= \\
\det
\big( \nabla_w(S^1(\fx, y_3)),
\nabla_w(S^2(\fx, y_3)), \nabla_w (\inn{B^\intercal\mu}{S_{y_3}(\fx,y_3)})\big) \det 
\frac{D\fw}{Dx}(\fx)
\end{multline*}
with $\fx\equiv \fx(w)\equiv \fx(w,a,b)$.
We have
$$\nabla_w \fS^1(\fx(w), z_3)=\nabla_w (S^1(\fx(w), b+z_3)),$$
and, with $b_{22}(z_3)= 1-\rho_2z_3$, 
$$\nabla_w \fS^2(\fx(w), z_3)=b_{22}(z_3) \nabla_w (S^2(\fx(w), b+z_3)) - 
b_{21}(z_3) \nabla_w (S^1(\fx(w), b+z_3));$$
moreover
$$\nabla_w \fS^1_{z_3}(\fx(w), b+z_3)
=\nabla_w (S^1_{y_3}(\fx(w), b+z_3)),$$
and
\begin{multline*}\nabla_w \fS^2_{z_3}(\fx(w), b+z_3)=
-\rho_1 \nabla_w (S^1(\fx(w), b+z_3))- 
\rho_2 \nabla_w (S^2(\fx(w), b+z_3))
\\+ b_{21} (z_3) \nabla_w (S^1_{y_3}(\fx(w), b+z_3))
+ b_{22} (z_3) \nabla_w (S^1_{y_3}(\fx(w), b+z_3)).
 \end{multline*}
 A quick calculation with determinants and \eqref{Detwx} yields the asserted identity \eqref{Deltarelation}.

For the curvature calculation we start with the equation \eqref{SversusfrakS} for the second component  and differentiate with respect to $w_3$. This yields
\begin{multline*}
\fS_{w_3}^2(w,z_3) =\\
(-\rho_3-\rho_1z_3) \big (\frac{D\fx}{Dw}e_3\big)^\intercal S^1_x(\fx,b+z_3)
+ (1-\rho_2z_3)
\big (\frac{D\fx}{Dw}e_3\big)^\intercal S^2_x(\fx,b+z_3).
\end{multline*}
here  the Jacobian $D\fx/Dw$ is evaluated at $(w,b)$.
We differentiate twice with respect to $z_3$ to obtain
\begin{align*}
&\fS_{w_3z_3z_3}^2(w,z_3) =
-2\rho_1 \big (\frac{D\fx}{Dw}e_3\big)^\intercal S^1_{xy_3}(\fx,b+z_3)
 -2\rho_2
\big (\frac{D\fx}{Dw}e_3\big)^\intercal S^2_{xy_3}(\fx,b+z_3)
\\&
-(\rho_3+\rho_1z_3) \big (\frac{D\fx}{Dw}e_3\big)^\intercal S^1_{xy_3y_3}(\fx,b+z_3)
+ (1-\rho_2z_3)
\big (\frac{D\fx}{Dw}e_3
\big)^\intercal S^2_{xy_3y_3}(\fx,b+z_3)
\end{align*}
where $\fx\equiv \fx(w,a,b)$. 
Using Cramer's rule we find that
\[
\frac{D\fx}{Dw}e_3 \big|_{(w,a,b)}  = \frac{1}{\Delta_1(\fx(w,a,b),b)}S^1_x\wedge S^2_x\big|_{(\fx(w,a,b),b)}.
\]
We evaluate the previous identity at $z_3=0$, and $w=0$
to obtain
\begin{align*}
\fS_{w_3z_3z_3}^2(0,0) =&\frac{1}{\Delta_1(a,b)}\Big( 
-2\rho_1 \inn {S^1_x\wedge S^2_x}{S^1_{xy_3}}
-2\rho_2
\inn{S^1_x\wedge S^2_x}{ S^2_{xy_3}}
\\&
-\rho_3 
\inn{S^1_x\wedge S^2_x}{ S^1_{xy_3y_3}}
+ \inn{S^1_x\wedge S^2_x}{ S^2_{xy_3y_3}}\Big|_{(a,b)} \Big)
\end{align*}
Using  \eqref{defofrho} we see that $\fS_{w_3z_3z_3}^2(0,0)$ equals
\begin{align*}
&\frac{1}{\Delta_1} 
\big(-2\rho_1 \Delta_1-2\rho_2\Delta_2- \rho_3(\Delta_{1,y_3}-\Gamma_1)+ (\Delta_{2,y_3}-\Gamma_2) \big)\Big|_{(a,b)}
\\=&
\frac{1}{\Delta_1^2}\big(2\Gamma_2\Delta_1- 2\Gamma_1\Delta_2-
(\Delta_{1,y_3}-\Gamma_1)\Delta_2+(\Delta_{2,y_3}-\Gamma_2)\Delta_1\big)\Big|_{(a,b)}
\end{align*}
which equals $\kappa(a,b)/\Delta_1(a,b)^2$ so that \eqref{curvature} is proved.
\end{proof}

\section{Decoupling in the general case}\label{decouplingstepsect}
 
 We consider the operator $\cR_{k.\ell}$ as in \eqref{Rkldef}. With  $\chi$ is as in \eqref{defofchik}  we  assume that $\chi$ is zero if $x\notin [-r_2/2, r_2/2]$ or if $y_3\notin [r_3/2, r_3/2]$ (see the paragraph leading to \eqref{xdefin} and  \eqref{r2def}, \eqref{r3def}).
 
 \begin{proposition}
  \label{decouplingstepprop}
  Let $0<\epsilon<1/2$, $\ell\le \floor{k/3}$.
 Let $\delta_0, \delta_1 \in (2^{-\ell (1-\epsilon^2)}, 2^{-\ell\epsilon^2} )$  such that \Be \max \{ (2^{-\ell} \delta_0)^{1/2}, \delta_0^{3/2} \}< \delta_1<\delta_0.\Ee
 Let $J$ be an interval of length $\delta_0$, near $b^\circ$,  and let $\cI_J$ be a collection of intervals of length $\delta_1$ which have disjoint interior and which intersect $J$. For each $I$, let $f_I$ be defined by $f_I(y)=f(y)\bbone_I(y_3)$. Let $a\in \bbR^3$, $\eps_0= (10M)^{-4}$,  $\vth\in C^\infty_c$ supported in $(-r_2,r_2)^3$ and $\vth_{\ell,a}(x)=\vth (2^{\ell}\eps_0^{-1}(x-a)).$
 Then
for $2\le p\le 6$,
\Be\label{decouplingstepest}
\Big\| \vth_{\ell,a} \sum_{I\in \cI_J} \cR_{k,\ell} f_I\Big\|_p 
\le C_\epsilon (\delta_0/\delta_1)^\epsilon
\Big(\sum_{I\in \cI_J} \big\|\vth_{\ell,a}\cR_{k,\ell} f_I \big\|_p^2\Big)^{1/2} + C_{N,\epsilon} 2^{-kN} \|f\|_p.
\Ee
The constants do not depend on the choice of $J$ and $\cI_J$.
\end{proposition}
\begin{proof}
Fix $a$ near $a^\circ$ and $b\in J$. We apply 
\eqref{SversusfrakS} and then  the changes of variables  
$y=\fy(z,a,b)$ in \eqref{fydef} and $\tau= {B^\intercal}^{-1}(z_3,a,b))\mu$.
Note from \eqref{Detzy}, \eqref{defofB} that $\det(D\fy/Dz)\det B=1$. Let $f(y)=\sum_{I\in \cJ_I} f\bbone_I(y_3)$ and $g(z,a,b)= f(\fy(z,a,b))$.

Let 
\Be \label{M1eps0}M_1\ge 1+\sum_{i=1}^2\sup_{(a,b)\in [-r_0,r_0]^4}\|\fS^i(\cdot,a,b)\|_{C^5([-r_0,r_0]^4)}.
\Ee which is just the uniform version of the condition  \eqref{M0eps0}. By applications of H\"older's inequality
it suffices to prove \eqref{decouplingstepest} under a slightly more restrictive assumptions on $\delta_0, \delta_1$, namely
\begin{align*}&\delta_0, \delta_1 \in (M_1^2 2^{20-\ell (1-\epsilon^2)}, 2^{-\ell\epsilon^2-20} M_1^{-2})\\& 2^{100} M_1\max \{ (2^{-\ell} \delta_0)^{1/2}, \delta_0^{3/2} \}< \delta_1<\delta_0.
\end{align*}
These are the uniform versions of \eqref{delta0delta1assu} which will allow us to apply Proposition \ref{decouplingstepmodelprop}. 

We have
\begin{multline*}
\cR_{k,\ell} f(x)
= 2^{2k}\iint e^{i2^k \inn{\mu}{\fS(w(x,a,b), z_3)-z'} }\tilde\chi_{k,\ell}(x,z,\mu, a,b)
 g(z,a,b) d\mu dz
\end{multline*}
with
\begin{multline*}
\tilde\chi_{k,\ell}(x,z,\mu, a,b)\,=\,
\chi(x,\fy(z,a,b))\eta_1(|{B^\intercal}^{-1}(z_3,a,b)\mu|) \times\\\eta \big(2^\ell \tfrac{\Delta_1(x)}{1-\rho_3(a,b) z_3} (\mu_1\Delta_1^\fS(w,z_3,a,b)+\mu_2\Delta_2^\fS(w,z_3,a,b))\big). 
\end{multline*}
Hence we get, with $\vsig_{\ell,a}(w) := \vth_{\ell,a} (\fx(w,a,b))$,
\[\vth_{\ell,a} (\fx(w,a,b))
\sum_I \cR_{k,\ell} f_I(\fx (w,a,b))=\vsig_{\ell,0} (w)\sum_{I\in\cJ_I} \cT_{k,\ell,a,b} g_I(w)
\]
where $g_I(z,a,b)= g(z,a,b)\bbone_{-b+I} (z_3)$ and $\cT_{k,\ell}\equiv \cT_{k,\ell,a,b}$ is as in \eqref{tildeRkl}.

We can now write the  left hand side of \eqref{decouplingest} as 
\begin{align*}
&\Big(\int\Big| \vth_{\ell,a}(\fx(w,a,b))  \sum_{I\in \cI_J} \cR_{k,\ell} f_I(\fx(w,a,b))\Big|^p |\det (\tfrac{D\fx}{Dw}| 
dw\Big)^{1/p}
\\
&\lc \Big\|\vsig_{\ell,0}  \sum_{I\in \cI_J} \cT_{k,\ell} g_I \Big\|_p
\end{align*} where we used uniform  upper bounds on $|\det (\tfrac{D\fx}{Dw})| $. 
By  Proposition \ref{decouplingstepmodelprop} we can bound
\[ \Big\|\vsig_{\ell,0}  \sum_{I\in \cI_J} \cT_{k,\ell} g_I \Big\|_p\le C_\epsilon (\delta_0/\delta_1)^\ep 
\Big(\sum_{I\in \cI_J} \big\|\vsig_{\ell,0}\cT_{k,\ell} g_I \big\|_p^2\Big)^{1/2} + C_{\epsilon} 2^{-10k} \|g\|_p.
\]
Undoing the above change of variable 
(and using 
uniform  lower  bounds on $|\det (\tfrac{D\fx}{Dw})| $) 
we may bound this, using Proposition
\ref{decouplingstepmodelprop}, 
 by
\[C_\epsilon' (\delta_0/\delta_1)^\ep 
\Big(\sum_{I\in \cI_J} \big\|\vth_{\ell,a}\cR_{k,\ell} f_I \big\|_p^2\Big)^{1/2} + C_{\epsilon} 2^{-10k} \|f\|_p.\qedhere
\]
\end{proof}

\begin{proof}[Proof of Theorem \ref{decouplingthm}]
We may assume $\eps<1/10$.
Let $\vth\in C^\infty_c(\bbR^3)$ supported in $(-1,1)^3$ such that $\vth\ge 0$ everywhere and $\sum_{n\in \bbZ^3}\vth(\cdot-n)=1$.

Let, for $n\in \bbZ^3$,
 $\zeta^{\ell,n}(x)= \ups(x)\zeta(2^\ell\eps_0^{-1}x-n)$.
 Thus
 \Be\|\ups\cR_{k,\ell} f\|_p \lc \Big(\sum_{n\in \bbZ^3}
 \big\| \vth^{\ell,n} \cR_{k,\ell} f\big\|_p^p\Big)^{1/p}.
 \Ee
 
 Now let $\cI(m)$ be the family of dyadic intervals with length $2^{-m}$.
 Let $I'$ be a dyadic interval of length $\ge 2^{-m}$ then we denote by
 $\cI(m,I')$ the collection of dyadic intervals which are of length $2^{-m}$ and are contained in $I'$.
 For any dyadic interval define 
 $f_I(y)= f(y)\bbone_{I}(y_3)$. 
 Let $m_0= \floor{\ell\eps^2}$.
 By H\"older's inequality,
 \Be\label{zerocaseind}
 \big\| \vth^{\ell,n} \cR_{k,\ell} f\big\|_p
 \le 2^{m_0(1-\frac 1p)}
 \Big(\sum_{J\in \cI(m_0)} \big\|\vth^{\ell,n}\cR_{k,\ell} f_J\big\|_p^p\Big)^{1/p}
 \Ee
 
 It is not hard to see that we can pick a sequence of integers
 $$m_1,\dots, m_{N(\ell)}$$  such that $m_j\le m_{j+1}\le\ell$
 for $j=0,\dots, N(\ell)-1$, and such that
 \begin{subequations}\label{mjseq}
 \Be m_{j+1}\le \min \big\{ \floor{\tfrac {3 m_j}{2}}, \floor{\tfrac{m_j+\ell}{2}} \big\};\Ee
 moreover \Be 
 \label{N(ell)def}
 m_{N} \ge \floor {\ell(1-\eps^2)}, \text{ and } 
  N(\ell)\le C_\eps\log_2(\ell).\Ee
  \end{subequations}

We claim that for $j=0,\dots,N(\ell)-1$  
\begin{multline}\label{jcaseind}\big\| \vth^{\ell,n} \cR_{k,\ell} f\big\|_p
 \le  C_{\eps^2}^{j} 2^{m_0(1-\frac 1p)+ (m_j-m_0)(\frac 12-\frac 1p+\eps^2)} 
 \Big(\sum_{I\in \cI(m_j)} \big\|\vth^{\ell,n}\cR_{k,\ell} f_I\big\|_p^p\Big)^{1/p}
  \\+ 2^{-9k}\Big(\sum_{\nu=0}^{j-1} C_{\eps^2}^\nu\Big )\|f\|_p
 \end{multline}
 We show this by induction. The case $j=0$ is covered by \eqref{zerocaseind}. For the induction step assume 
 \eqref{jcaseind} for some $j<N(\ell)-1$.
  Observe that for $I\in \cI(m_j)$ Proposition \ref{decouplingstepprop} and H\"older's inequality give
 \begin{align*}&\big\|\vth^{\ell,n}\cR_{k,\ell} f_I\big\|_p
 =\Big\|\vth^{\ell,n}\sum_{I'\in \cI(m_{j+1},I)} 
 \cR_{k,\ell} f_{I'}\big\|_p
 \\ &\le
 C_{\eps^2} 2^{(m_{j+1}-m_j)\eps^2}
  \Big(\sum_{I'\in \cI(m_{j+1},I)} \|\vth^{\ell,n}\cR_{k,\ell} f_{I'}\big\|_p^2\Big)^{1/2} + C_{\eps^2} 2^{-10 k} \|f_I\|_p
  \\&\le
 C_{\eps^2} 2^{(m_{j+1}-m_j)(\frac 12-\frac 1p+\eps^2)}
  \Big(\sum_{I'\in \cI(m_{j+1},I)} \|\vth^{\ell,n}\cR_{k,\ell} f_{I'}\big\|_p^p\Big)^{1/p}+C_{\eps^2} 2^{-10 k} \|f_I\|_p
 \end{align*}
We use the induction hypothesis  \eqref{jcaseind} and by the last inequality we  bound
$\big\| \vth^{\ell,n} \cR_{k,\ell} f\big\|_p$ by
\begin{align*} 
& C_{\eps^2}^{j+1} 2^{m_0(1-\frac 1p)+ (m_{j+1}-m_0)(\frac 12-\frac 1p+\eps^2)} 
 \Big(\sum_{I'\in \cI(m_{j+1})} \big\|\vth^{\ell,n}\cR_{k,\ell} f_{I'}\big\|_p^p\Big)^{1/p}
  \\&+ C_{\eps^2}^j 
  2^{m_0(1-\frac 1p)+ (m_{j+1}-m_0)(\frac 12-\frac 1p+\eps^2)} 2^{-10 k} \Big(\sum_{I\in \cI(m_{j})} \big \|f_I\big\|_p^p\Big)^{1/p}
  \\&+2^{-9k}\Big(\sum_{\nu=0}^{j-1} C_{\eps^2}^\nu\Big)\|f\|_p.
  \end{align*}
  Since 
$m_0(1-\frac 1p)+ (m_{j+1}-m_0)(\frac 12-\frac 1p+\eps^2)\le k$ and 
$(\sum_{I\in \cI(m_{j})} \big \|f_I\big\|_p^p)^{1/p}\le \|f\|_p$ we obtain the case for $j+1$ of \eqref{jcaseind}.

We consider the case $j=N(\ell)$ of \eqref{jcaseind}.
Observe that each interval $I\in \cI(m_{N(\ell)})$ is the union of 
$2^{\ell-m_{N(\ell)}}$ dyadic intervals of 
length $2^{-\ell}$. We also sum in $n\in\bbZ^3$ and use the finite overlap of the supports of $\zeta^{\ell,n}$. Observe that the cardinality of the index set of $n$ which give a nonzero contribution is $O(2^{3\ell})=O(2^k)$.
We get 
\begin{align*}&\big\|  \cR_{k,\ell} f\big\|_p
\lc \Big(\sum_{n\in \bbZ^3}\|\zeta^{\ell,n} 
\cR_{k,\ell} f\big\| _p^p\Big)^{1/p}
\\
& \le  C_{\eps^2}^{j} 2^{m_0(1-\frac 1p)+ (m_j-m_0)(\frac 12-\frac 1p+\eps^2)+(\ell-m_{N(\ell)})(1-\frac 1p)}
 \Big(\sum_{I\in \cI(\ell)} \sum_{n\in \bbZ^3}\big\|\zeta^{\ell,n} \cR_{k,\ell} f_I\big\|_p^p\Big)^{1/p}
  \\&+ 2^{-8k}\Big(\sum_{\nu=0}^{N(\ell)-1} C_{\eps^2}^\nu\Big)
  \|f\|_p.
 \end{align*}
Observe
\begin{multline*}m_0(1-\frac 1p)+ (m_j-m_0)(\frac 12-\frac 1p+\eps^2)+(\ell-m_{N(\ell)})(1-\frac 1p)\\
\le \ell( \eps^2(1-\frac 1p)+(1-2\eps^2)(\frac 12-\frac 1p+\eps^2)\le \ell(2\eps^2+1/2-1/p)\end{multline*} and
(with $N(\ell)$ as in \eqref{N(ell)def})
\[\sum_{l=0}^{N(\ell)-1} C_{\eps^2}^l \lc (1+\ell)^{B(\eps)}\] for some large constant $B(\eps)$.
This yields the assertion of the theorem.
\end{proof}
\section{$L^p$-Sobolev estimate} \label{endpoint-bound}
In order to prove our Sobolev estimate we have to combine the estimates for the operators $\cR_k$.
Here we use a special case of Theorem 1.1. in \cite{prs}. In what follows the operators $P_k$ are 
defined by $\widehat {P_k f}(\xi)=\phi(2^{-k}\xi)\widehat f$,
where $\phi$ is supported in $\{\xi:\tfrac 12<|\xi|<2\}$

\begin{proposition} (\cite{prs})
Assume $\eps>0$, $p_0<p<\infty$ and $\la>1$. We are given   operators $T_k$, $k>0$, with smooth Schwartz kernels $K_k$  (acting on functions in $\bbR^3$) satisfying
\begin{subequations}\begin{align}\label{Tkassu1}
&\sup_{k>0} 2^{k/p} \|T_k\|_{L^p\to L^p} \le A 
\\ \label{Tkassu2}
&\sup_{k>0} 2^{k/p_0} \|T_k\|_{L^{p_0}\to L^{p_0}} \le B_0. 
\end{align} Assume that for each cube $Q$ there is a measurable exceptional set $E_Q$ such that \Be\label{Tkassumeas}\meas(E_Q)\le \la\max \{\diam(Q)^2, |Q|\}\Ee
 and such that for every $k>0$ and every cube $Q$ with $2^k\diam(Q)\ge 1$ we have
\Be \label{Tkassu3} \int_{\bbR^3\setminus E_Q} |K_k(x,y)|dy \le B_1 \max 
\{(2^k\diam(Q)^{-\eps}, 2^{-k\eps}\} \text{ for a.e. $x\in Q$.}
\Ee
\end{subequations}
Then for $q>0$, 
\begin{multline}\label{Lpresult}\Big\|\Big(\sum_{k>0} 2^{kq/p} |P_kT_k f_k|^q\Big)^{1/q} \Big\|_p\\
\lc A \Big[ \log \big(3+\frac{ B_0^{\frac{p_0}{p}} (A\la^{1/p}+B_1)^{1-\frac{p_0}{p}}}{A}\big) \Big]^{\frac 1q-\frac 1p}
\Big(\sum_k\|f_k\|_p^p\Big)^{\frac 1p}.
\end{multline}
\end{proposition}

We claim that for $\ell>0$
\Be
\label{Sobresult}\Big\|\Big(\sum_{k:\floor{k/3}\ge\ell} 2^{kq/p} |P_k\cR_{k,\ell} f_k|^q\Big)^{1/q} \Big\|_p\\
\le C_p  2^{-\ell\eps(p)} 
\Big(\sum_k\|f_k\|_p^p\Big)^{\frac 1p}, \quad p>4
\Ee
which can be used, together with \eqref{Littlewood-Paley-calc}  to deduce
\Be\label{curveresult}
 \cR:  (B^s_{p,p})_{\text{comp}} \to (F^{s+1/p}_{p,q})_{\loc}, \quad p>4, \,\, q>0. \Ee
Since $L^s_p= F^s_{p,2} \hookrightarrow B^s_{p,p}$ for $p>2$ and 
$F^0_{p,q}\hookrightarrow F^0_{p,2} =L^p$, $q\le 2$,  this implies the asserted $L^p$-Sobolev estimates.
In order to check \eqref{Sobresult} we need to verify the assumptions of the proposition for the family $\{\cR_{k,\ell}\}_{k\ge 3\ell}$.

Let $4<p_0<p$. By Theorem \ref{Rklthm} we have \eqref{Tkassu1} with $A=C_p2^{-\ell \beta}$ and 
$\beta<2/p-1/2$ if $4<p\le 6$ and $\beta<1/p$ if $p\ge 6$. Moreover
we have \eqref{Tkassu2} with $B_0=C_{p_0} 2^{-\ell \beta_0}$ and 
$\beta_0<2/p_0-1/2$ if $4<p_0\le 6$ and $\beta_0<1/p$ if $p_0\ge 6$. 

 By  integration by parts argument one has the bound 
 $$|R_{k,\ell}(x,y)|\le C_N 
 \frac{2^{2k}}{(1+2^{k-\ell} |y'- S(x_Q,y_3)|
)^N}\,
$$ for the Schwartz kernel of $\cR_{k,\ell}$.
For a cube  $Q$ with center
$x_Q$ define 
$$E_Q:=\{y: |y'- S(x_Q,y_3)|\le C 2^{2\ell}\diam(Q)\}$$
if $\diam(Q)\le 1$. If $\diam (Q)\ge 1$ we let $E_Q$ be a ball
 of diameter $C2^{2\ell}\diam(Q)$,  centered at $x_Q$. Assumption \eqref{Tkassumeas} is then satisfied with the choice of $\la= 2^{2\ell}$ and \eqref{Tkassu3} holds with $B_1= 2^{2\ell}$. The logarithmic term in \eqref{Lpresult} gives us  an additional factor  $O(\ell)$. Thus we have verified \eqref{Sobresult} with $\eps(p)<\beta$ and 
\eqref{curveresult} follows by summation in $\ell\ge 0.$
 
 \section{Further results and conjectures}\label{lastsection}
 In our analysis we heavily used the condition $\ell\le\floor{k/3}$ for the operators $\cR_{k,\ell}$. If one is interested to relax the assumption that $\pi_R$ is a fold one needs to explore finer localizations of
  $\tau_1\Delta_1+\tau_2\Delta_2$ as used by Comech in \cite{comech}. There he proves sharp $L^2$-Sobolev estimates under the assumption that $\pi_L$ is a fold but $\pi_R$  satisfies a finite type condition of order 
  $\ft$, i.e. if $V_R$ is a kernel field for $\pi_R$ then $\sum_{j=0}^{\ft} |V_R^j(\det\pi_R)|\neq 0$. 
  The case $\ft=1$ applies to the fold assumption on $\pi_R$. 
   In the general  finite type situation we can show the 
  $L^p_{\comp} \to L^p_{1/p,\loc}$ estimate for $p\ge 5$, and in fact in a slightly larger range.

\begin{thm}\label{ftthm}
Let $\cM\subset \Om_L\times\Om_R$  be a four-dimensional manifold  such that the projections \eqref{singsupp} are submersions. Assume  that  
 the only singularities of $\pi_L:(\cN^*\cM)'\to T^*\Om_L$  are Whitney folds and that 
  $\pi_R:(\cN^*\cM)'\to T^*\Om_R$ is of finite type $\le \ft$, for some $\ft\ge 0$.  With $\cL, \varpi$ be as in Theorem \ref{mainthm} suppose
  that $\varpi$  is a submersion. Then $\cR$ is extends to a continuous  operator 
$$ \cR: L^p_{\comp} (\Omega_R)\to  L^p_{1/p,\loc}(\Omega_L), \quad \tfrac{10\ft+2}{2\ft+1}<p<\infty\,.
$$
\end{thm}
\begin{proof}[Sketch of Proof]
 By the $L^2$ estimates in \cite{comech} the operators $\cR_{\ell,k}$ the $L^2$ bound in \eqref{phstest} is  still valid, and all of our previous arguments apply. Hence we just need to consider the case $\ell=\floor{k/3}$.
 
 The operator $\cR_{\floor{k/3} ,k}$,   for which 
  $|\tau_1\Delta_1+\tau_2\Delta_2|\lc 2^{-k/3}$,  
 satisfies the norm estimate   $\|\cR_{k,\lfloor k/3\rfloor} \|_{L^2\to L^2}
  \lc 2^{-\frac{k}{2} \frac{\ft+1}{2\ft+1}}$, a less satisfactory bound.  
  One can show this estimate as a consequence of more refined $L^2$-estimates in \cite{comech}.
  This yields an  analogue of \eqref{ellpLpbds} in the finite type case, namely for $2\le p\le\infty$,
  \Be\Big(\sum_\nu \big\|\cR_{k,\floor{k/3}} [\bbone_{\floor{k/3},\nu}g_\nu] \big\|^p\Big)^{\frac 1p} 
\lc 
2^{-\frac kp\frac{\ft+1}{2\ft+1} -\frac k3(1-\frac 2p)}\Big(\sum_\nu  \|g_\nu\|_p^p\Big)^{\frac 1p}.
\notag
\Ee 
Combining this with the  decoupling estimate
\eqref{decouplingest} (which remains true for $\ell=\floor{k/3}$) yields 
\Be
\big\|  \cR_{k,\floor{k/3}} f
\big\|_p 
\lc C_\eps 2^{\frac{k}{3} (\frac 12-\frac 1p+\ep)}
2^{-\frac kp\frac{\ft+1}{2\ft+1} -\frac k3(1-\frac 2p)}
\|f\|_p, 
\text{  $2\le p\le 6$},
\notag
\Ee
i.e.   
$\|  \cR_{k,\floor{k/3}} f
\|_p \lc 2^{-k(\alpha(p)+1/p)}$ with $\alpha(p)>0$ for $\frac{10\ft+2}{2\ft+1}< p\le 6$. 
Further  interpolation with the  bound $\|\cR_{k,\floor{k/3}}\|_{L^\infty\to L^\infty} =O(1)$ gives a similar statement for  $6\le p\le \infty$ with an $\alpha(p)>0$ for $6\le p<\infty$. 
\end{proof}

  To improve on this result, one would have to employ finer localizations in terms of $\det \pi_L$
  (which would correspond to the assumption 
  $|\tau_1\Delta_1+\tau_2\Delta_2| \approx 2^{-\ell}$ where a range of $\ell>k/3$ will depend on $\ft$). Our current arguments for the  plate localization  in  Lemma \ref{platelocalizationlemma} are  not effective 
in that situation. Nevertheless we conjecture that the result 
 of Theorem \ref{ftthm} remains true for all $p>4$, and even that the assumptions on $\pi_R$ can be dropped altogether in Theorem \ref{mainthm}. 
 See the discussion of  model examples in \S\ref{xraysubsection} and \S\ref{momentcurveHsect}.

\end{document}